\documentclass[11pt]{article}

\usepackage{amssymb}
\usepackage{amsmath}
\usepackage{graphicx}
\usepackage{caption}
\usepackage{subcaption}
\usepackage{caption}
\usepackage{float}
\usepackage{subcaption}
\usepackage{dcolumn}
\usepackage{amsthm}
\usepackage{multirow}
\usepackage{enumitem}
\usepackage{amsfonts}
\usepackage{color}
\usepackage[dvipsnames]{xcolor}

\usepackage[sort]{cite}

\usepackage{pgf}
\usepackage{tikz}
\usetikzlibrary{decorations.pathmorphing}
\usetikzlibrary{arrows.meta,decorations.pathmorphing,backgrounds,positioning,fit}
\usetikzlibrary{shapes,snakes}
\usetikzlibrary{calc}
\tikzset{snake it/.style={decorate, decoration=snake}}

\numberwithin{equation}{section}

\def\Line(#1,#2)(#3,#4){\qbezier(#1,#2)(#1,#2)(#3,#4)}

\newcommand{\inp}[2][]{\left(#1, #2\right)}
\newcommand{\gnp}[2][]{\langle#1, #2\rangle}

\newtheorem{remark}{Remark}[section]
\newtheorem{lemma}{Lemma}[section]

\newtheorem{theorem}{Theorem}[section]
\newtheorem{definition}{Definition}[section]

\def\Tc{\mathcal{T}}

\def\Mc{\mathcal{M}}

\def\BDM{\mathcal{BDM}}
\def\BDDF{\mathcal{BDDF}}

\def\Pc{\mathcal{P}}

\def\X{\mathbb{X}}
\def\W{\mathbb{W}}
\def\R{\mathbb{R}}
\def\M{\mathbb{M}}
\def\S{\mathbb{S}}
\def\N{\mathbb{N}}

\def\H{\mathbb{H}}

\def\s{\sigma}
\def\t{\tau}
\def\g{\gamma}
\def\c{\chi}

\def\z{\zeta}
\def\tet{\theta}
\def\del{\delta}

\def\r{\mathbf{r}}

\def\As{A_{\s\s}}
\def\Au{A_{\s u}}
\def\Ag{A_{\s\g}}

\def\Eh{\hat{E}}

\def\Mh{\hat{M}}
\def\rh{\hat{\r}}
\def\eh{\hat{e}}

\def\xh{\hat{x}}
\def\yh{\hat{y}}

\def\th{\hat{\t}}

\def\Xh{\hat{\X}}
\def\Vh{\hat{V}}

\def\asym{\operatorname{as\,}}
\def\tr{\operatorname{tr\,}}

\def\Quh{\mathcal{R}^{u}_h}
\def\Qgh{\mathcal{R}^{\g}_h}

\def\skew{\operatorname{Skew}}

\def\curl{\operatorname{curl}}
\def\dvr{\operatorname{div}}          
\def\dvrg{\operatorname{div}}  

\def\O{\Omega}
\def\Gn{\Gamma_N}
\def\Gd{\Gamma_D}

\usepackage{geometry}
\begin{document}

\title{A multipoint stress mixed finite element method for elasticity on simplicial grids}

\author{Ilona Ambartsumyan\thanks{Oden Institute for Computational Engineering and Sciences, The University of
		Texas at Austin, Austin, TX 78712, USA;~{\tt \{ailona@austin.utexas.edu, ekhattatov@austin.utexas.edu,\}}. Partially supported by DOE grant DE-FG02-04ER25618 and NSF grants DMS 1418947 and DMS 1818775.}~\and
		Eldar Khattatov\footnotemark[1]\and
		Jan M. Nordbotten\footnotemark[2]\thanks{Department of Mathematics, University of Bergen, Bergen, 7803, Norway;~{\tt \{Jan.Nordbotten@uib.no\}}. Funded in part through Norwegian Research Council grants 250223, 233736, 228832.}\and
		Ivan Yotov\footnotemark[2]\thanks{Department of Mathematics, University of
		Pittsburgh, Pittsburgh, PA 15260, USA;~{\tt \{yotov@math.pitt.edu\}}. Partially supported by DOE grant DE-FG02-04ER25618 and NSF grants DMS 1418947 and DMS 1818775.}~}

\date{\today}
\maketitle
\begin{abstract}
We develop a new multipoint stress mixed finite element method for
linear elasticity with weakly enforced stress symmetry on simplicial
grids. Motivated by the multipoint flux mixed finite element method
for Darcy flow, the method utilizes the lowest order
Brezzi-Douglas-Marini finite element spaces for the stress and the
vertex quadrature rule in order to localize the
interaction of degrees of freedom. This allows for local stress
elimination around each vertex.  We develop two variants of the
method. The first uses a piecewise constant rotation and results in a
cell-centered system for displacement and rotation. The second uses a
piecewise linear rotation and a quadrature rule for the asymmetry
bilinear form. This allows for further elimination of the rotation,
resulting in a cell-centered system for the displacement
only. Stability and error analysis is performed for both
variants. First-order convergence is established for all variables in
their natural norms. A duality argument is further employed to prove
second order superconvergence of the displacement at the cell
centers. Numerical results are presented in confirmation of the
theory.
\end{abstract}

\section{Introduction}
Mixed finite element (MFE) methods for elasticity with
stress-displacement formulations provide accurate stress, local
momentum conservation, and robust treatment of almost incompressible
materials. Numerous methods with strong stress symmetry
\cite{ArnWin-elast,arnold2005rectangular,brezzi2008mixed} and weak
stress symmetry
\cite{arnold1984peers,stenberg1988family,arnold2007mixed,boffi2009reduced,
  cockburn2010new,gopalakrishnan2012second,arnold2015mixed,Awanou-rect-weak,lee2016towards}
have been developed. A drawback of these methods is that the resulting
algebraic systems are of saddle point type. {\color{black} Two
  common approaches for reducing MFE formulations to positive definite
  systems include hybridization, resulting in skeletal systems, and
  reduction to cell-centered systems.  In the context of stress-displacement elasticity formulations,
  hybridization is possible for non-conforming MFE methods
  \cite{Arnold-Winther-elast-nonconf-2d,Arnold-etal-elast-nonconf-3d,GG-elast-nonconf} or 
hybridizable discontinuous Galerkin (HDG) methods
  \cite{Cockburn-Shi-HDG-elast,Qui-etal-HDG-elast}. These methods require facet displacement degrees of
  freedom corresponding to polynomials of at least first order.}

{\color{black} In this paper we develop MFE methods for elasticity on simplicial
grids that can be reduced to symmetric and positive definite cell
centered systems based on piecewise constant
  approximations. These methods have reduced computational complexity
compared to hybrid formulations, both due to the smaller polynomial degree and the 
fact that there are fewer elements than facets.}
  Our approach is motivated by the multipoint flux
mixed finite element (MFMFE) method
\cite{wheeler2006multipoint,Ing-Whe-Yot,WheXueYot-nonsym} for Darcy
flow, which is closely related to the multipoint flux approximation
(MPFA) method
\cite{aavat2002introduction,aavat1998discretization,edwards2002unstructured,edwards1998finite,agelas2010convergence}.
The MPFA method is a finite volume method obtained by eliminating
fluxes around mesh vertices in terms of neighboring pressures. It can
handle discontinuous full tensor coefficients and general grids, thus
improving on previously developed cell centered finite difference
methods resulting from MFE methods
\cite{russell1983finite,arbogast1997mixed,ADKWY}, which work for
smooth grids and/or coefficients. The MFMFE method
\cite{wheeler2006multipoint,Ing-Whe-Yot} utilizes the lowest order
Brezzi-Douglas-Marini $\BDM_1$ spaces \cite{brezzi1985two} on
simplicial and quadrilateral grids, see also a similar approach in
\cite{Brezzi.F;Fortin.M;Marini.L2006} on simplices, as well as an
enhanced Brezzi-Douglas-Duran-Fortin $\BDDF_1$ space
\cite{brezzi1991mixed} on hexahedra. An alternative formulation based
on a broken Raviart-Thomas velocity space is developed in
\cite{klausen2006robust}. A common feature of the above mentioned
methods is that the velocity space has only degrees of freedom that
are normal components of the vector on the element boundary, such that
on any facet (edge or face) there is one normal velocity associated
with each of the vertices. An application of the vertex quadrature
rule for the velocity bilinear form results in localizing the
interaction of velocity degrees of freedom around mesh vertices. The
fluxes then can be locally eliminated, resulting in a cell centered
pressure system. The variational framework of the MFMFE methods allows
for combining MFE techniques with quadrature error analysis to
establish stability and convergence results.

In \cite{Jan-IJNME}, the multipoint stress approximation (MPSA) method
for elasticity was developed, which is a displacement finite volume
method based on local stress elimination around vertices in a manner
similar to the MPFA method. The method does not have a mixed finite
element interpretation, but its stress degrees of freedom correspond
to the $\BDM_1$ degrees of freedom. The MPSA method was analyzed in
\cite{nordbotten2015convergence} by being related to a non-symmetric
discontinuous Galerkin (DG) method. A weak symmetry MPSA method has
been developed in \cite{keilegavlen2017finite}.

In this paper we develop two stress-displacement MFE methods for
elasticity on simplicial grids that reduce to cell centered
systems. We consider the formulation with weakly imposed stress
symmetry, for which there exist MFE spaces with $\BDM_1$ degrees of
freedom for the stress and piecewise constant displacement. In this
formulation the symmetry of the stress is imposed weakly using a
Lagrange multiplier, which is a skew-symmetric matrix and has a
physical meaning of rotation.  Our first method is based on the
Arnold-Falk-Winther (AFW) spaces \cite{arnold2007mixed}
$\BDM_1\times\Pc_0\times\Pc_0$, i.e., $\BDM_1$ stress and piecewise
constant displacement and rotation. Since in $\R^d$ there are $d$
normal stress vector degrees of freedom per facet, one degree of
freedom can be associated with each vertex, We employ the vertex
quadrature rule for the stress bilinear form, which localizes the
stress degrees of freedom interaction around vertices, resulting in a
block-diagonal stress matrix.  {\color{black} This approach resembles
  the well-known mass-lumping procedure.}  The stress is then locally
eliminated and the method is reduced to a symmetric and positive
definite cell centered system for the displacement and rotation. This
system is smaller and easier to solve than the original saddle point
problem, but no further reduction is possible.  {\color{black}Our
  second method is based on the modified AFW triple
  $\BDM_1\times\Pc_0\times\Pc_1$ proposed in \cite{brezzi2008mixed}.}
The difference from the first method is that the rotation is
continuous piecewise linear. In this method we employ the vertex
quadrature rule both for the stress and the asymmetry bilinear
forms. This allows for further local elimination of the rotation after
the initial stress elimination, resulting in a symmetric and positive
definite cell centered system for the displacement only. This is a
very efficient method with computational cost comparable to the MPSA
method. Adopting the MPSA terminology, we call our method a multipoint
stress mixed finite element (MSMFE) method, with the two variants
referred to as MSMFE-0 and MSMFE-1, where the number corresponds to
the rotation polynomial degree.

{\color{black} We note that the MSMFE methods inherit the locking-free
  property of stress-displacement MFE methods for elasticity. A
  numerical example illustrating the robustness of the MSMFE methods
  for nearly-incompressible materials is presented in the numerical
  section.}  {\color{black} We should mention that a number of
  locking-free primal-formulation methods have also been developed,
  including DG \cite{riviere2000optimal}, a hybrid high order 
  method \cite{dipietro2015hybrid}, finite element methods with the
  Crouzeix-Raviart space \cite{Brenner-Sung,dipietro2015extension,
    Hansbo-Larson}, weak Galerkin methods \cite{WG-elast}, a virtual
  element method \cite{VEM-elast}, and a hybrid finite volume method
  \cite{DiPietro-FV-elasticity}. These methods have been developed on
  general polygonal grids, and many of them can have arbitrary order
  of approximation. However, they require postprocessing for computing
  the stress and do not provide local equilibrium with
  $H(\dvr)$-conforming stress.}
 
We perform stability and error analysis for both MSMFE methods. The
stability analysis follows the framework established in previous works
on MFE methods for elasticity with weak stress symmetry
\cite{arnold2007mixed,arnold2015mixed}, and utilizes the classical
Babu\v{s}ka-Brezzi conditions \cite{brezzi1991mixed}.  While the
stability of the MSMFE-0 method is relatively straightforward, the
analysis of the MSMFE-1 method is not. It requires establishing an
inf-sup condition for the Taylor-Hood Stokes pair with vertex
quadrature in the divergence bilinear form. We do this by employing a
macroelement argument, following an approach developed in
\cite{stenberg1984analysis}. We note that our analysis differs from
the one in \cite{stenberg1984analysis}. In particular, the application
of the vertex quadrature rule leads to additional technical
difficulties in the inf-sup analysis, since the control of the
pressure degrees of freedom by the velocity in the divergence bilinear
form is weakened. We proceed with establishing first order convergence
for the stress in the $H(\dvr)$-norm and for the displacement and
rotation in the $L^2$-norm for both methods. The arguments combine
techniques from MFE analysis and quadrature error analysis. A duality
argument is further employed to prove second order superconvergence of
the displacement at the cell centers.

The rest of the paper is organized as follows. The model problem and
its MFE approximation are presented in Section~2.  The two methods are
developed and their stability is analyzed in Sections 3 and 4,
respectively. Section 5 is devoted to the error analysis. Numerical
results are presented in Section 6.

\section{Model problem and its MFE approximation}
In this section we recall the weak stress symmetry formulation of the
elasticity system. We then present its MFE approximation and a
quadrature rule, which form the basis for the MSMFE methods presented
in the next sections.

Let $\O$ be a simply connected bounded domain of $\R^d,\, d=2,3$
occupied by a linearly elastic body. We write $\M$, $\S$ and $\N$ for
the spaces of $d\times d$ matrices, symmetric matrices and
skew-symmetric matrices, all over the field of real numbers,
respectively. The material properties are described at each point $x
\in \O $ by a compliance tensor $A = A(x)$, which is a symmetric,
bounded and uniformly positive definite linear operator acting from
$\S$ to $\S$. We also assume that an extension of $A$ to an operator
$\M \to \M$ still possesses the above properties. As an example, in the
case of a homogeneous and isotropic body,
$$ A\sigma = \frac{1}{2\mu} \left( \sigma 
- \frac{\lambda}{2\mu + d\lambda}\operatorname{tr}(\sigma)I \right), 
$$
where $I$ is the $d \times d$ identity matrix and $\mu > 0$, $\lambda > -2\mu/d$
are the Lam\'{e} coefficients.

Throughout the paper the divergence operator is the usual divergence
for vector fields. When applied to a matrix field, it produces a
vector field by taking the divergence of each row. We will also use
the curl operator which is the usual curl when applied to vector
fields in three dimensions, and it is defined as
\begin{align*}
    \curl{\phi} = (\partial_2 \phi, -\partial_1 \phi) 
\end{align*}
for a scalar function $\phi$ in two dimensions. Consequently, for a
vector field in two dimensions or a matrix field in three dimensions,
the curl operator produces a matrix field, by acting row-wise.

Throughout the paper, $C$ denotes a generic positive constant that is independent of the discretization parameter $h$. We will also use the following standard notation. For a domain $G \subset \R^d$, the $L^2(G)$ inner product and norm for scalar and vector valued
functions are denoted $\inp[\cdot]{\cdot}_G$ and $\|\cdot\|_G$, respectively. The norms and seminorms of the Sobolev spaces $W^{k,p}(G),\, k \in \R, p>0$ are denoted by $\| \cdot \|_{k,p,G}$ and $| \cdot |_{k,p,G}$, respectively. The norms and seminorms of the Hilbert spaces $H^k(G)$ are denoted by $\|\cdot\|_{k,G}$ and $| \cdot |_{k,G}$, respectively. We omit $G$ in the subscript if $G = \O$. For a section of the domain or element boundary $S \subset \R^{d-1}$ we write $\gnp[\cdot]{\cdot}_S$ and $\|\cdot\|_S$ for the $L^2(S)$ inner product (or duality pairing) and norm, respectively. We will also use the spaces 
\begin{align*}
    &H(\dvrg; \O) = \{v \in L^2(\O, \R^d) : \dvr v \in L^2(\O)\}, \\
    &H(\dvrg; \O, \M) = \{\t \in L^2(\O, \M) : \dvr \t \in L^2(\O,\R^d)\},
\end{align*}    
equipped with the norm
$$\|\t\|_{\dvr} = \left( \|\t\|^2 + \|\dvr\t\|^2 \right)^{1/2}.$$

Given a vector field $f$ on $\Omega$ representing body forces, equations of static elasticity in Hellinger-Reissner form determine the stress $\sigma$ and the displacement $u$ satisfying the constitutive and equilibrium equations respectively:
\begin{align}
    A\s = \epsilon(u), \quad \dvr \s = f \quad \text{in } \O, \label{elast-1}
\end{align}
together with the boundary conditions 
\begin{align}
    u = g \ \text{ on } \Gd, \quad  \s\,n = 0 \ \text{ on } \Gn, \label{elast-2}
\end{align}
where $\epsilon(u) = \frac{1}{2}\left(\nabla u + (\nabla u)^T\right)$
and $\partial\O = \Gd \cup \Gn$. We assume for simplicity that $\Gd \neq \emptyset$.

Introducing the Lagrange multiplier $\g = \skew(\nabla u)$,
$\skew(\tau) = \frac12(\tau - \tau^T)$, from the space of
skew-symmetric matrices to penalize the asymmetry of the stress
tensor, and using that $A\sigma = \nabla u - \gamma$,
we arrive at the weak formulation for
\eqref{elast-1}-\eqref{elast-2}, see for example
\cite{arnold2007mixed,arnold1984peers}: find $(\s, u, \gamma) \in
\X \times V \times \W$ such that
\begin{align} 
\begin{aligned}\label{weak-elast}
	\inp[A\s]{\t} + \inp[u]{\dvr \t} + \inp[\g]{\t} &= \gnp[g]{\t\, n}_{\Gd}, &\forall \t &\in \X,  \\
	\inp[\dvr \s]{v} &= \inp[f]{v}, &\forall v &\in V, \\
	\inp[\s]{\xi} &= 0, &\forall \xi &\in \W,
\end{aligned}
\end{align}
where the spaces are
$$\X = \big\{ \t\in H(\dvrg;\Omega,\M) : \t\,n = 0 \text{ on } \Gn  \big\}, \quad V = L^2(\Omega, \R^d), \quad \W = L^2(\Omega, \N).$$

Define the asymmetry operator 
$\asym:\M \to \R^{d(d-1)/2}$ such that
$$
\asym(\tau) = \tau_{12} - \tau_{21} \mbox{ in 2D} \,\,\mbox{ and } 
\asym(\tau) = (\tau_{32} - \tau_{23}, \tau_{31} - \tau_{13}, 
\tau_{21} - \tau_{12})^T \mbox{ in 3D}.
$$ 
Let 
$$
\H = \left\{ \begin{array}{l} \R^2, \, d = 2, \\ \M, \, d = 3, \end{array} \right.
$$
and define the invertible operators $S: \H \to \H$ and 
$\Xi: \R^{d(d-1)/2} \to \N$ as follows,
\begin{align}
\begin{aligned} \label{skew-extra}
&d=2: &&S(w) = w             \quad \mbox{for } w\in \R^2, & & \Xi(p) = \begin{pmatrix} 0 & p \\ -p & 0 \end{pmatrix} \quad \mbox{for } p\in \R, \\
&d=3: &&S(w) = \tr(w) I - w^T \quad \mbox{for } w\in \M, & & \Xi(p) = \begin{pmatrix} 0 & -p_3 & p_2 \\ p_3 & 0 & -p_1 \\ -p_2 & p_1 & 0 \end{pmatrix} \quad \mbox{for } p\in\R^3.
\end{aligned}
\end{align}
A direct calculation shows that for all $w \in \H$,
\begin{equation}\label{curl-div}
\quad \asym(\curl(w)) = - \dvr S(w),
\end{equation}
and for all $\t\in \M$ and $\xi\in \N$,
\begin{align}\label{asym-identity}
\inp[\t]{\xi} = \inp[\asym(\t)]{\Xi^{-1}(\xi)}.
\end{align}


\subsection{Mixed finite element method}
Here we present the MFE approximation of \eqref{weak-elast}, which is
the basis for the MSMFE methods.  Assume for simplicity that $\O$ is a
polygonal domain and let $\Tc_h$ be a shape-regular and quasi-uniform
finite element partition of $\O$ \cite{ciarlet2002finite} consisting
of triangles in two dimensions or tetrahedra in three dimensions with
maximum diameter $h$. For any element $E \in \Tc_h$ there exists a
bijection mapping $F_E: \Eh \to E$, where $\Eh$ is a reference
element. Denote the Jacobian matrix by $DF_E$ and let $J_E =
|\operatorname{det} (DF_E)|$.  In the case of triangular meshes, $\Eh$
is the reference right triangle with vertices $\rh_1 = (0,0^T)$,
$\rh_2 = (1,0)^T$ and $\rh_3 = (0,1)^T$. Let $\r_1$, $\r_2$ and $\r_3$
be the corresponding vertices of $E$, oriented counterclockwise. In
this case $F_E$ is a linear mapping of the form $F_E(\rh) = \r_1(1-\xh
- \yh) + \r_2\xh + \r_3\yh$ with a constant Jacobian matrix and
determinant given by $DF_E = [\r_{21}, \r_{31}]^T$ and $J_E = 2|E|$,
where $\r_{ij} = \r_i - \r_j$.  The mapping for tetrahedra is
described similarly.

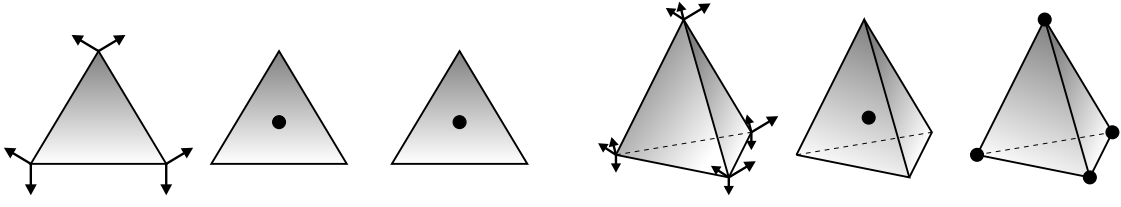
\begin{figure}
\setlength{\unitlength}{1.0mm}
\begin{center}
\scalebox{.6}{
\begin{tikzpicture}
    [inner sep=1mm,
    dof/.style={circle,draw=black,fill=black,thick}]
    \draw [line width=1.5pt,arrows = {-Stealth[inset=0,width=7pt,length=7pt]}] (2,3) -- (2.6,3.36);
    \draw [line width=1.5pt,arrows = {-Stealth[inset=0,width=7pt,length=7pt]}] (2,3) -- (1.4,3.36);
    \draw [line width=1.5pt,arrows = {-Stealth[inset=0,width=7pt,length=7pt]}] (3.5,0.5) -- (4.1,0.86);
    \draw [line width=1.5pt,arrows = {-Stealth[inset=0,width=7pt,length=7pt]}] (3.5,0.5) -- (3.5,-0.2);
    \draw [line width=1.5pt,arrows = {-Stealth[inset=0,width=7pt,length=7pt]}] (0.5,0.5) -- (-0.1,0.86);
    \draw [line width=1.5pt,arrows = {-Stealth[inset=0,width=7pt,length=7pt]}] (0.5,0.5) -- (0.5,-0.2);
    \shadedraw [very thick,shading=axis] (0.5,0.5) -- (3.5,0.5) -- (2.0,3.0) -- cycle;
    \shadedraw [very thick,shading=axis] (4.5,0.5) -- (7.5,0.5) -- (6.0,3.0) -- cycle;
    \shadedraw [very thick,shading=axis] (8.5,0.5) -- (11.5,0.5) -- (10,3.0) -- cycle;
    \node at (6,1.425) [dof] {};
    \node at (10,1.425) [dof] {};
\end{tikzpicture}
}
\hspace{.5cm}
\scalebox{.6}{
\begin{tikzpicture}
    [inner sep=1mm,
    dof/.style={circle,draw=black,fill=black,thick}]
    \draw [line width=1.5pt,arrows = {-Stealth[inset=0,width=6pt,length=6pt]}] (1.0,1.0) -- (0.9,1.4);
    \draw [line width=1.5pt,arrows = {-Stealth[inset=0,width=6pt,length=6pt]}] (2.5,4.0) -- (2.4,4.4);
    \draw [line width=1.5pt,arrows = {-Stealth[inset=0,width=6pt,length=6pt]}] (4.0,1.5) -- (3.9,1.9);
    \draw [very thick,shading=axis,shading angle=45] (3.5,0.5) -- (4.0,1.5) -- (2.5,4.0);
    \draw [very thick,shading=axis,shading angle=45] (1.0,1.0) -- (3.5,0.5) -- (2.5,4.0) -- (1.0,1.0);
    \draw [dashed] (1.0,1.0) -- (4.0,1.5);
    \draw [line width=1.5pt,arrows = {-Stealth[inset=0,width=7pt,length=7pt]}] (2.5,4.0) -- (3.1,4.36);
    \draw [line width=1.5pt,arrows = {-Stealth[inset=0,width=7pt,length=7pt]}] (3.5,0.5) -- (4.1,0.86);
    \draw [line width=1.5pt,arrows = {-Stealth[inset=0,width=7pt,length=7pt]}] (4.0,1.5) -- (4.6,1.86);
    \draw [line width=1.5pt,arrows = {-Stealth[inset=0,width=6pt,length=6pt]}] (1.0,1.0) -- (0.6,1.26);
    \draw [line width=1.5pt,arrows = {-Stealth[inset=0,width=6pt,length=6pt]}] (2.5,4.0) -- (2.1,4.26);
    \draw [line width=1.5pt,arrows = {-Stealth[inset=0,width=6pt,length=6pt]}] (3.5,0.5) -- (3.1,0.76);
    \draw [line width=1.5pt,arrows = {-Stealth[inset=0,width=6pt,length=6pt]}] (1.0,1.0) -- (1.0,0.6);
    \draw [line width=1.5pt,arrows = {-Stealth[inset=0,width=6pt,length=6pt]}] (4.0,1.5) -- (4.0,1.1);
    \draw [line width=1.5pt,arrows = {-Stealth[inset=0,width=6pt,length=6pt]}] (3.5,0.5) -- (3.5,0.1);
    \draw [very thick,shading=axis,shading angle=45] (7.5,0.5) -- (8.0,1.5) -- (6.5,4.0);
    \draw [very thick,shading=axis,shading angle=-15] (5.0,1.0) -- (7.5,0.5) -- (6.5,4.0) -- (5.0,1.0);
    \draw [dashed] (5.0,1.0) -- (8.0,1.5);
    \draw [very thick,shading=axis,shading angle=45] (11.5,0.5) -- (12.0,1.5) -- (10.5,4.0);
    \draw [very thick,shading=axis,shading angle=-15] (9.0,1.0) -- (11.5,0.5) -- (10.5,4.0) -- (9.0,1.0);
    \draw [dashed] (9.0,1.0) -- (12.0,1.5);
    \node at (6.6,1.825) [dof] {};

    \node at (11.5,0.5) [dof] {};
    \node at (12,1.5) [dof] {};
    \node at (10.5,4.0) [dof] {};
    \node at (9,1.0) [dof] {};

\end{tikzpicture}
}
\end{center}
\caption{$\BDM_1\times\Pc_0\times\Pc_0$ on triangles (left) and 
$\BDM_1\times\Pc_0\times\Pc_1$ on tetrahedra (right).
}
\label{fig:elements}
\end{figure}

The finite element spaces $\X_h \times V_h \times \W_h^{k} 
\subset \X \times V \times \W$ are the triple
$ \left(\BDM_1\right)^d \times \left(\Pc_0\right)^d \times
\left(\Pc_k\right)^{d\times d, skew}$, where $k = 0,1$. Note that 
for $k=1$ the space $\W_h^1$ contains continuous piecewise linears.
On the reference triangle these spaces are defined as
\begin{align}
    \hat{\X}(\Eh) = \Pc_1(\Eh)^2 \times \Pc_1(\Eh)^2 = 
\begin{pmatrix} \alpha_1 \xh + \beta_1 \yh + \gamma_1  
& \alpha_2 \xh + \beta_2 \yh + \gamma_2 \\ 
\alpha_3 \xh + \beta_3 \yh + \gamma_3  
& \alpha_4 \xh + \beta_4 \yh + \gamma_4 \end{pmatrix}, \nonumber \\
    \hat{V}(\Eh) = \Pc_0(\Eh) \times \Pc_0(\Eh),  \quad
    \hat{\W}^k(\Eh) = \Xi(p) ,\, p\in \Pc_k(\Eh) \mbox{ for } k = 0,1.  
\label{spaces-with-j}
\end{align}
The definition on tetrahedra is similar, except that 
$\hat{\W}^k(\Eh) = \Xi(p)$, $p\in \Pc_k(\Eh)^d$. These spaces satisfy
$$
\dvr \Xh(\Eh) = \Vh(\Eh) \quad \mbox{and} \quad
\forall \th \in \Xh (\Eh),\,\eh \in \Eh, \quad \th\, n_{\eh} \in \Pc_1(\eh)^d.
$$
It is known \cite{brezzi1985two,brezzi1991mixed} that the degrees of
freedom for $\BDM_1$ can be chosen as the values of normal fluxes
at any two points on each edge $\eh$ if $\Eh$ is a reference triangle,
or any three points one each face $\eh$ if $\Eh$ is a reference
tetrahedron. This naturally extends to normal stresses in the case of
$(\BDM_1)^d$. Here we choose these points to be at the
vertices of $\eh$, see Figure~\ref{fig:elements}. This choice is motivated by the
use of quadrature rule described in the next section.  The spaces on
any element $E \in \Tc_h$ are defined via the transformations
\begin{align*}
 \t \overset{\Pc}\leftrightarrow \hat{\t} : \t^T = \frac{1}{J_E} DF_E \hat{\t}^T \circ F_E^{-1}, 
\quad
    v \leftrightarrow \hat{v} : v = \hat{v} \circ F_E^{-1}, \quad
    \xi \leftrightarrow \hat{\xi} : \xi = \hat{\xi} \circ F_E^{-1},
\end{align*}
where $\t \in \X$, $v \in V$, and $\xi \in \W$.  The stress tensor is transformed
by the Piola transformation applied row-wise. It
preserves the normal components of the stress tensor
on facets, and it satisfies 
\begin{equation}
    (\dvr \t, v)_E = (\dvr \hat{\t}, \hat{v})_{\Eh} \quad \text{and} \quad \langle \t\, n_e, v \rangle _e = \langle \hat{\t}\, \hat{n}_{\eh}, \hat{v} \rangle _{\eh}. \label{prop-1}
\end{equation}
The spaces on $\Tc_h$ are defined by
\begin{align} \label{mixed-spaces}
    \X_h &= \{\t \in \X: \t|_E \overset{\Pc}\leftrightarrow \hat{\t},\: \hat{\t} \in \hat{\X}(\Eh) \quad \forall E\in\mathcal{T}_h\}, \nonumber \\
    V_h &= \{v \in V: v|_E \leftrightarrow \hat{v},\: \hat{v} \in \hat{V}(\Eh) \quad \forall E\in\mathcal{T}_h\}, \\
{\color{black}
    \W_h^0} & {\color{black} = \{\xi \in \W: \xi|_E \leftrightarrow \hat{\xi},
\: \hat{\xi} \in \hat{\W}^0(\Eh) \quad \forall E\in\mathcal{T}_h\}}, \nonumber \\
{\color{black}    \W_h^1} & {\color{black} = \{\xi \in \mathcal{C}(\O,\N) \subset \W: \xi|_E \leftrightarrow \hat{\xi},\: \hat{\xi} \in \hat{\W}^1(\Eh) \quad \forall E\in\mathcal{T}_h\}.} \nonumber
\end{align}
Note that $\W_h^1 \subset H^1(\Omega)$, since it contains continuous
piecewise $\Pc_1$ functions. The mixed finite element approximation of \eqref{weak-elast} is:
find $(\sigma_h, u_h, \g_h) \in \X_h \times V_h \times \W_h^k$ such that
\begin{align}
	(A\sigma_h,\tau) + (u_h,\dvr{\tau}) + (\g_h, \tau) &= \gnp[g]{\tau\,n}_{\Gd}, &\tau &\in \X_h, \\
	(\dvr \sigma_h,v) &= (f,v), &v &\in V_h, \\
	(\sigma_h,\xi) &= 0, &\xi &\in \W_h^k.
\end{align}
The method has a unique solution and it is first order accurate for
all variables in their corresponding norms with both choices of
rotation elements, see \cite{arnold2007mixed} for $k=0$ and
\cite{cockburn2010new} for $k=1$. A drawback is that the resulting
algebraic system is of a saddle point type and couples all three
variables. We next present a quadrature rule that
allows for local eliminations of the stress in the case of
$k=0$, resulting in a cell-centered displacement-rotation system in
the case $k=0$.  In the case $k=1$, a further elimination of the
rotation can be performed, which leads to a displacement-only 
cell-centered system.

\subsection{A quadrature rule}

Let $\varphi$ and $\psi$ be continuous functions on $\Omega$. We
denote by $(\varphi,\psi)_Q$ the application of the element-wise vertex
quadrature rule for computing $(\varphi,\psi)$. In particular, for
$\chi,\,\t \in \X_h$, we have
$$ 
(A\chi,\tau)_Q = \sum_{E \in \Tc_h} (A\chi,\tau)_{Q,E} = 
\sum_{E \in \Tc_h} \frac{|E|}{s}\sum_{i=1}^{s} A\chi(\r_i):\tau(\r_i),
$$ 
where $s=3$ on triangles and $s=4$ on tetrahedra. The vertex tensor
$\chi(\r_i)$ is uniquely determined by its normal components
$(\chi\,n_{ij})(\r_i)$, $j = 1,\ldots,d$, where $n_{ij}$ are the
outward unit normal vectors on the two edges (three faces) that share
$\r_i$, see Figure~\ref{fig:elements}. More precisely, $\chi(\r_i) =
\sum_{j=1}^{d} (\chi\, n_{ij})(\r_i)n_{ij}^T$.  Since the basis
functions associated with a vertex are zero at all other vertices, the
quadrature rule decouples $(\chi\,n_{ij})(\r_i)$ from the rest of the
degrees of freedom, which allows for local stress elimination.

We also employ the quadrature rule for the stress-rotation bilinear form 
in the case of linear rotations. For $\tau = \X_h,\, \xi \in \W^1_h$ we have
\begin{equation*}
(\tau,\xi)_{Q,E} = \frac{|E|}{s}\sum_{i=1}^{s} \tau(\r_i):\xi(\r_i).
\end{equation*}
Again, only degrees of freedom associated with a vertex are coupled, which
allows for further elimination of the rotation. 

For $\chi,\, \t \in \X_h$ and $\xi \in \W_h^1$ denote the element 
quadrature errors by
\begin{align} \label{quad-err-def}
    \tet_E(A\chi,\t) = (A\chi,\tau)_E - (A\chi,\tau)_{Q,E}, \quad
    \del_E(\t,\xi) = (\t,\xi)_E - (\t,\xi)_{Q,E}.
\end{align}
and define the global quadrature errors by 
$\theta(A\chi,\tau)|_E = \tet_E(A\chi,\tau)$, $\del(\t,\xi)|_E = \del_E(\t,\xi)$. 

\begin{lemma} \label{0-q-err-const}
If $\chi \in \X_h(E)$ and $\xi \in \W^1_h(E)$, then for 
all constant tensors $\tau_0$ and for all skew-symmetric constant tensors 
$\z_0$,
$$ \tet_E(\chi,\tau_0) = 0, \quad \del_E(\chi,\z_0) = 0, \quad
\del_E(\tau_0,\xi) = 0.
$$
\end{lemma}
\begin{proof}
It is enough to consider $\tau_0$ such that it has only one nonzero
component, say, $(\tau_0)_{1,1} = 1$; the arguments for other cases
are similar. Since the quadrature rule $(f)_{Q,E} = \frac{|E|}{s}
\sum_{i=1}^s f(\r_i)$ is exact for linear functions, we have
	\begin{equation*}
		\inp[\chi]{\tau_0}_{Q,E} = \frac{|E|}{s}\sum_{i=1}^{s}(\chi)_{1,1}(\r_i) = \int_E \chi : \tau_0\, dx.
	\end{equation*}
	The same reasoning applies for the other statements.
\end{proof}

\begin{lemma}\label{quad-inner-prod}
The bilinear form $\inp[A\tau]{\chi}_Q$ is an inner product on $\X_h$
and $\inp[A\tau]{\tau}_Q^{1/2}$ is a norm in $\X_h$ equivalent to
$\|\cdot \|$, i.e., there exist constants $0 < \alpha_0 \le \alpha_1$
independent of $h$ such that
\begin{equation}\label{norm-equiv}
\alpha_0 \|\tau\|^2 \le \inp[A\tau]{\tau}_Q \le \alpha_1 \|\tau\|^2 \quad 
\forall \tau\in\X_h.
\end{equation}
Furthermore, $(\xi,\xi)^{1/2}_Q$ is a norm in $\W^1_h$ equivalent to
$\|\cdot\|$, and $\forall \, \tau \in \X_h$, $\xi \in \W^1_h$,
$(\tau,\xi)_Q \le C \|\tau\|\|\xi\|$.

\end{lemma}
\begin{proof}
The properties of the operator $A$ imply that there exist positive
constants $a_0$ and $a_1$ such that for all $\tau \in \M$, 
$a_0 \, \tau:\tau \le A\tau:\tau \le a_1 \, \tau:\tau$.
Let $\tau = \sum_{i=1}^s \sum_{j=1}^d
\tau_{ij}\chi_{ij} $ on an element $E$, where $\chi_{ij}$ are basis
functions as shown in Figure \ref{fig:elements}. We have 
\begin{align*}
\inp[A\tau]{\tau}_{Q,E} = \frac{|E|}{s} \sum_{i=1}^s A\tau(\r_i) : \tau(\r_i) 
\ge a_0\frac{|E|}{s} \sum_{i=1}^s \tau(\r_i) : \tau(\r_i) 
\ge C |E| \sum_{i=1}^s \sum_{j=1}^d \tau_{ij}^2.
\end{align*}
On the other hand,
\begin{equation*}
\| \tau \|^2_E = \inp[\sum_{i=1}^s \sum_{j=1}^d \tau_{ij}\chi_{ij}]{\sum_{k=1}^s \sum_{l=1}^d \tau_{kl}\chi_{kl}}_E \le C|E| \sum_{i=1}^s \sum_{j=1}^d \tau_{ij}^2,
\end{equation*}
which implies $ \alpha_0\| \tau \|^2 \le \inp[A\tau]{\tau}_Q$.  Since
$\inp[A\tau]{\chi}_Q$ is symmetric and linear, it is an inner product and
$\inp[A\tau]{\tau}_Q^{1/2}$ is a norm on $\X_h$. The upper bound 
in \eqref{norm-equiv} follows from a scaling argument, see 
\cite[Corollary 2.5]{wheeler2006multipoint}. The proofs of the other two 
statements are similar.
\end{proof}
%


\section{The multipoint stress mixed finite element method with constant rotations
(MSMFE-0)}
In the first method, referred to as MSMFE-0, we use the piecewise constant 
space $\W_h^0$ for rotations and apply the quadrature rule only to the
stress bilinear form.  The method is:
find $\sigma_h \in
\X_h,\, u_h \in V_h$ and $\g_h \in \W_h^0$ such that
\begin{align}
(A\sigma_h,\tau)_Q + (u_h,\dvr{\tau}) + (\g_h, \tau) &= \gnp[g]{\tau\,n}_{\Gd}, &\tau &\in \X_h, \label{msmfe-0-1} \\
(\dvr \sigma_h,v) &= (f,v), &v &\in V_h, \label{msmfe-0-2} \\
(\sigma_h,\xi) &= 0, &\xi &\in \W_h^0. \label{msmfe-0-3} 
\end{align}

\begin{theorem} \label{msmfe-stability-0}
The method \eqref{msmfe-0-1}--\eqref{msmfe-0-3} has a unique solution
$(\s_h,u_h,\g_h)$.
\end{theorem}
\begin{proof}
Using the classical stability theory of mixed finite element
methods \cite{brezzi1991mixed}, the solvability of 
\eqref{msmfe-0-1}--\eqref{msmfe-0-3} follows from 
the Babu\v{s}ka-Brezzi conditions:
 \begin{enumerate} [label={\bf(S\arabic*)},align=left]
\item \label{S1} There exists $c_1>0$ such that 
for all $\tau \in \X_h$ satisfying 
$\inp[\dvr \tau]{v} + \inp[\tau]{\xi} = 0$ for all $(v,\xi) \in V_h\times \W_h^0$,
$$ 
c_1\| \tau \|_{\dvrg}^2 \le \inp[A\tau]{\tau}_{Q}, 
$$
\item \label{S2} There exists $c_2 > 0$ such that 
$$ \inf_{0\neq(v,\xi)\in V_h\times \W_h^0 } \sup_{0\neq \tau\in\X_h} \frac{\inp[\dvr\tau]{v}+\inp[\tau]{\xi}}{\| \tau \|_{\dvrg} \left( \|v\| + \|\xi\| \right)}  \ge c_2.$$
 \end{enumerate} 
Condition \ref{S1} is satisfied due to Lemma~\ref{quad-inner-prod},
while condition \ref{S2} is shown in \cite{arnold2007mixed,brezzi2008mixed}.
\end{proof}


\subsection{Reduction to a cell-centered displacement-rotation system}

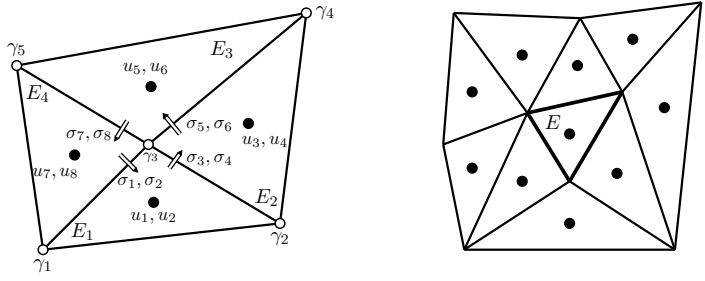
\begin{figure}
\begin{center}
\scalebox{.7}{
\begin{tikzpicture}
    [inner sep=0.65mm,
    displ/.style={circle,draw=black,fill=black,thick},
    rotat/.style={circle,draw=black,fill=white,thick}]
    \coordinate (p1) at (0,0);
    \coordinate (p2) at (4.5,0.5);
    \coordinate (p3) at (2.0,2.0); 
    \coordinate (p4) at (5.0,4.5);
    \coordinate (p5) at (-0.5,3.5);

    \draw [very thick] (p1) -- (p2);
    \draw [very thick] (p2) -- (p4);
    \draw [very thick] (p4) -- (p5);
    \draw [very thick] (p1) -- (p3);
    \draw [very thick] (p2) -- (p3);
    \draw [very thick] (p4) -- (p3);
    \draw [very thick] (p5) -- (p3); 
    \draw [very thick] (p5) -- (p1); 

    \draw [line width=0.7pt,arrows = {-Stealth[inset=0,width=4pt,length=3pt]}, double distance = 1.3pt] (2.6,2.2) -- (2.27,2.6);
    \draw [line width=0.7pt,arrows = {-Stealth[inset=0,width=3pt,length=3pt]}, double distance = 1.3pt] (1.6,2.45) -- (1.35,2.05);
    \draw [line width=0.7pt,arrows = {-Stealth[inset=0,width=3pt,length=3pt]}, double distance = 1.3pt] (1.46,1.81) -- (1.81,1.46);
    \draw [line width=0.7pt,arrows = {-Stealth[inset=0,width=3pt,length=3pt]}, double distance = 1.3pt] (2.4,1.55) -- (2.65,1.9);


    \node at (2.1,0.9) [displ] {};
    \node at (3.9,2.4) [displ] {};
    \node at (2.05,3.10) [displ] {};
    \node at (0.6,1.8) [displ] {};

    \node at (p1) [rotat] {};
    \node at (p2) [rotat] {};
    \node at (p3) [rotat] {};
    \node at (p4) [rotat] {};
    \node at (p5) [rotat] {};

    \draw (p1) node[below=3pt] {$\g_1$};
    \draw (p2) node[below=3pt] {$\g_2$};
    \draw (2.05,2.02) node[below=3pt] {\scriptsize$\g_3$};
    \draw (p4) node[right=3pt] {$\g_4$};
    \draw (p5) node[above=3pt] {$\g_5$};
    \draw (2.1,0.9) node[below=3pt] {\small$u_1,u_2$};
    \draw (3.6,2.05) node[right=3pt] {\small$u_3,u_4$};
    \draw (1.95,3.15) node[above=3pt] {\small$u_5,u_6$};
    \draw (0.85,1.45) node[left=3pt] {\small$u_7,u_8$};
    \draw (1.85,1.55) node[below=3pt] {\small$\s_1,\s_2$};
    \draw (2.55,1.70) node[right=3pt] {\small$\s_3,\s_4$};
    \draw (2.55,2.35) node[right=3pt] {\small$\s_5,\s_6$};
    \draw (1.49,2.17) node[left=3pt] {\small$\s_7,\s_8$};
    \draw (0.35,0.35) node[right=3pt] {$E_1$};
    \draw (4.25,0.65) node[above=3pt] {$E_2$};
    \draw (3.8,3.8) node[left=3pt] {$E_3$};
    \draw (-0.1,3.3) node[below=3pt] {$E_4$};
    \coordinate (q1) at (8,0);
    \coordinate (q2) at (12,0.0);
    \coordinate (q3) at (12.5,4.7); 
    \coordinate (q4) at (10.2,4.4);
    \coordinate (q5) at (7.8,4.5);
    \coordinate (q6) at (7.6,2.0);
    \coordinate (q7) at (10,1.3);
    \coordinate (q8) at (9.2,2.6);
    \coordinate (q9) at (11,3.0);

    \draw [very thick] (q1) -- (q2);
    \draw [very thick] (q2) -- (q3);
    \draw [very thick] (q3) -- (q4);
    \draw [very thick] (q4) -- (q5);
    \draw [very thick] (q5) -- (q6);    
    \draw [very thick] (q6) -- (q1); 
    \draw [line width=2pt] (q7) -- (q8);   
    \draw [line width=2pt] (q8) -- (q9);    
    \draw [line width=2pt] (q7) -- (q9);
    \draw [very thick] (q1) -- (q7);
    \draw [very thick] (q1) -- (q8);    
    \draw [very thick] (q2) -- (q7);    
    \draw [very thick] (q2) -- (q9);
    \draw [very thick] (q3) -- (q9);
    \draw [very thick] (q4) -- (q9);
    \draw [very thick] (q4) -- (q8);
    \draw [very thick] (q5) -- (q8);
    \draw [very thick] (q6) -- (q8);
    \node at (10,0.5) [displ] {};
    \node at (10.9,1.45) [displ] {};
    \node at (11.8,2.7) [displ] {};
    \node at (11.2,4.0) [displ] {};
    \node at (10.15,3.5) [displ] {};
    \node at (9.1,3.7) [displ] {};
    \node at (8.15,3.0) [displ] {};
    \node at (8.15,1.55) [displ] {};
    \node at (9.1,1.3) [displ] {};
    \node at (10,2.2) [displ] {};
    \draw (9.35,2.43) node[right=3pt] {$E$};
\end{tikzpicture}
}
\end{center}
\caption{Finite elements sharing a vertex (left) and displacement stencil (right)}
\label{fig:elimination}
\end{figure}

{\color{black}
The algebraic system that arises from 
\eqref{msmfe-0-1}--\eqref{msmfe-0-3} is of the form
\begin{equation}\label{sp-matrix}
    \begin{pmatrix} \As & \Au^T & \Ag^T \\ - \Au & 0   & 0 \\ - \Ag & 0   & 0 \end{pmatrix} 
    \begin{pmatrix} \sigma \\ u \\ \g \end{pmatrix} = 
    \begin{pmatrix} g \\ - f \\ 0 \end{pmatrix},
\end{equation}
}
{\color{black}
where $(\As)_{ij} = (A\tau_j,\tau_i)_Q$, $(\Au)_{ij} = (\dvr\tau_j,
v_i)$, and $(\Ag)_{ij} = (\tau_j,\xi_i)$. In the standard MFE
formulation without quadrature rule, all stress degrees of freedom are
coupled in the matrix $\As$ and it is not possible to eliminate the
stress with local computations, thus the entire saddle point problem
needs to be solved. In contrast, the MSMFE-0 method is designed to allow for 
local and inexpensive stress elimination, as shown below. }
{\color{black}
\begin{lemma}
The matrix $\As$ is block-diagonal with symmetric and positive definite 
blocks associated with the mesh vertices.
\end{lemma}
\begin{proof} 
Let us consider any interior vertex $\r$ and suppose that it is shared
by $k$ elements $E_1,\dots,E_k$ as shown in Figure
\ref{fig:elimination}. Let $e_1,\dots,e_k$ be the facets that
share the vertex $\r$ and let $\t_1,\dots,\t_{dk}$, be the stress
basis functions on these facets associated with the
vertex. Denote the corresponding values of the normal components of
$\s_h$ by $\s_1,\dots,\s_{dk}$. Note that for the sake of clarity the
normal stresses are drawn at a distance from the vertex.
As noted above, the quadrature rule $(A\cdot,\cdot)_Q$
localizes the basis functions interaction, therefore 
taking $\tau = \tau_1,\dots,\tau_{dk}$ in \eqref{msmfe-0-1}
results in a $d\,k\times d\,k$ local linear system
for $\s_1,\dots,\s_{dk}$, implying that $(\As)$ is block-diagonal
with $dk\times dk$ blocks associated with the mesh vertices. Furthermore,
$$ 
(A\sigma_h,\tau_i)_Q = \sum_{j=1}^{dk} \sigma_j(A\tau_j,\tau_i)_Q 
= \sum_{j=1}^{dk} (\As)_{ij}\sigma_j, \qquad i=1,...,dk,
$$
and by Lemma~\ref{quad-inner-prod}, each $dk\times dk$ block $(\As)_{ij}$, $i,j=1,...,dk$,
is symmetric and positive definite. 
\end{proof}
}
As a consequence of the above lemma,
$\sigma$ can be easily eliminated from \eqref{sp-matrix}, resulting in the 
displacement-rotation system
\begin{equation}\label{msmfe0-system}
    \begin{pmatrix} \Au\As^{-1}\Au^T & \Au\As^{-1}\Ag^T \\ 
\Ag\As^{-1}\Au^T & \Ag\As^{-1}\Ag^T \end{pmatrix} 
    \begin{pmatrix} u \\ \g \end{pmatrix} =
    \begin{pmatrix} \tilde{f} \\ \tilde{h} \end{pmatrix}.
\end{equation}

\begin{lemma}\label{disp-rot-spd}
The cell-centered displacement-rotation system \eqref{msmfe0-system} is 
symmetric and positive definite.
\end{lemma}
\begin{proof}
The symmetry of the matrix follows from the symmetry of $\As$.
To show the positive definiteness, for any
$\begin{pmatrix} v^T & \xi^T \end{pmatrix} \neq 0$, 
\begin{equation*}
\begin{pmatrix} v^T & \xi^T \end{pmatrix}
\begin{pmatrix} \Au\As^{-1}\Au^T & \Au\As^{-1}\Ag^T 
\\ \Ag\As^{-1}\Au^T & \Ag\As^{-1}\Ag^T \end{pmatrix}
\begin{pmatrix} v \\ \xi \end{pmatrix} 
= (\Au^Tv + \Ag^T\xi)^T\As^{-1}(\Au^Tv + \Ag^T\xi) > 0,
\end{equation*}
 due to the inf-sup condition \ref{S2}.
\end{proof}

\begin{remark}
The MSMFE-0 method is more efficient than the original MFE method,
since it reduces the initial saddle-point problem to a smaller
symmetric and positive definite cell-centered system for displacement
and rotation. However, further reduction in the system is not
possible, since the diagonal blocks in \eqref{msmfe0-system} couple
all displacement, respectively rotation, degrees of freedom and are
not easily invertible. In the next section we propose a method with
linear rotations and a vertex quadrature rule applied to the
stress-rotation bilinear forms. This allows for further local
elimination of the rotation, resulting in a cell-centered system for
displacement only.
\end{remark}
%

\section{The multipoint stress mixed finite element method with linear rotations
(MSMFE-1)}

In the second method, referred to as MSMFE-1, we use the continuous piecewise
linear space $\W_h^1$ for rotations and apply the quadrature rule to both the stress
bilinear form and the stress-rotation bilinear forms. The method is:
find $\sigma_h \in \X_h,\, u_h \in V_h$ and $\g_h \in \W_h^1$ such that
\begin{align}
	(A\sigma_h,\tau)_Q + (u_h,\dvr{\tau}) + (\g_h,\tau)_Q &= \gnp[g]{\tau}_{\Gd}, &\tau &\in \X_h, \label{msmfe-1-1}\\
	(\dvr \sigma_h,v) &= (f,v), &v &\in V_h, \label{msmfe-1-2} \\
	(\sigma_h,\xi)_Q &= 0, &\xi &\in \W_h^1. \label{msmfe-1-3}
\end{align}
{\color{black}
\begin{remark}\label{scaled-rotation}
We note that the rotation finite element space in the MSMFE-1 method
is continuous, which may result in reduced approximation if the
rotation $\gamma \in L^2(\Omega,\N)$ is discontinuous.  It is possible
to consider a modified MSMFE-1 method based on the scaled rotation
$\tilde \gamma = A^{-1} \gamma$, which is motivated by the relation
$\sigma = A^{-1}\nabla u - A^{-1}\gamma$. This method is better suited
for problems with discontinuous compliance tensor $A$, since in this
case $\sigma$ is smoother than $A\sigma$, implying that $\tilde
\gamma$ is smoother than $\gamma$. The resulting method is: find
$\sigma_h \in \X_h,\, u_h \in V_h$ and $\tilde\g_h \in \W_h^1$ such
that
\begin{align}
	(A\sigma_h,\tau)_Q + (u_h,\dvr{\tau}) + (\tilde\g_h,A\tau)_Q &= \gnp[g]{\tau}_{\Gd}, &\tau &\in \X_h, \label{scaled-msmfe-1-1}\\
	(\dvr \sigma_h,v) &= (f,v), &v &\in V_h, \label{scaled-msmfe-1-2} \\
	(A\sigma_h,\xi)_Q &= 0, &\xi &\in \W_h^1. \label{scaled-msmfe-1-3}
\end{align}
In the numerical section we present an example with discontinuous $A$
and $\gamma$ illustrating the advantage of the modified method
\eqref{scaled-msmfe-1-1}--\eqref{scaled-msmfe-1-3} for problems with
discontinuous coefficients. In order to maintain uniformity of the
presentation in relation to MSMFE-0, as well as conformity with the
standard formulation for weakly symmetric MFE methods for elasticity
used in the literature, in the following we present the well-posedness
and error analysis for the method
\eqref{msmfe-1-1}--\eqref{msmfe-1-3}. We note that the analysis for
the modified method \eqref{scaled-msmfe-1-1}--\eqref{scaled-msmfe-1-3}
is similar.
\end{remark}
}
The stability conditions for the MSMFE-1 method are as follows:
 \begin{enumerate} [label={\bf(S\arabic*)},align=left]
    \setcounter{enumi}{2}
	\item \label{S3} 
There exists $c_3>0$ such that 
for all $\tau \in \X_h$ satisfying 
$\inp[\dvr \tau]{v} + \inp[\tau]{\xi} = 0$ for all $(v,\xi) \in V_h\times \W_h^1$,
$$ 
c_3\| \tau \|_{\dvrg}^2 \le \inp[A\tau]{\tau}_{Q}, 
$$
\item \label{S4} There exists $c_4 > 0$ such that 
$$ 
\inf_{0\neq(v,\xi)\in V_h\times \W_h^1 } \sup_{0\neq \tau\in\X_h} 
\frac{\inp[\dvr\tau]{v}
+\inp[\tau]{\xi}_Q}{\| \tau \|_{\dvrg} \left( \|v\| + \|\xi\| \right)}  \ge c_4. 
$$
 \end{enumerate} 

\subsection{Well-posedness of the MSMFE-1 method}
While the coercivity condition \ref{S3} is again satisfied due to
Lemma~\eqref{quad-inner-prod}, we need to verify the inf-sup condition
\ref{S4}. The difficulty is due to the quadrature rule in
$\inp[\tau]{\xi}_Q$. The next theorem, which is a modification of
\cite[Theorem 3.2]{arnold2015mixed}, provides sufficient conditions
for a triple $\X_h\times V_h \times \W^1_h$ to satisfy \ref{S4}.

\begin{theorem}\label{msmfe-stability-1}
Let $S_h\subset H(\dvr; \Omega)$ and $U_h \subset L^2(\Omega)$ be a
stable mixed Darcy pair, i.e., there exists $c_5 > 0$ such that
\begin{align}
\inf\limits_{0\neq r\in U_h} \sup\limits_{0\neq z \in S_h} 
\frac{\inp[\dvr z]{r}}{\| z \|_{\dvr} \|r\| }  \ge c_5, \label{darcy-pair}
\end{align} 
and let $Q_h \subset H^1(\Omega,\H)$ and $W_h \subset L^2(\Omega,\R^{d(d-1)/2})$ be
a stable mixed Stokes pair, such that $(w,w)_Q^{1/2}$ is a norm in $W_h$
equivalent to $\|w\|$ and there exists $c_6 > 0$ such that
\begin{align}
\inf\limits_{0\neq w\in W_h} \sup\limits_{0\neq q \in Q_h} 
\frac{\inp[\dvr q]{w}_Q}{\| q \|_1 \|w\| }  \ge c_6. \label{inf-sup-mod}
\end{align} 
Suppose further that 
\begin{align} \label{curl-Q-S}
\curl Q_h \subset (S_h)^d.
\end{align}
Then, $\X_h = (S_h)^d \subset H(\dvr;\Omega,\M)$,
$V_h = (U_h)^d \subset L^2(\Omega, \R^d)$, and $\W_h^1 = \Xi(W_h) \subset
L^2(\Omega,\N)$ satisfy \ref{S4}.
\end{theorem}
\begin{proof}
Let $v \in V_h, \,w \in W_h$ be given. It follows from \eqref{darcy-pair} that
there exists $\eta \in \X_h$ such that
\begin{equation}\label{eta-eqn}
(\dvr \eta, v) =\|v\|^2, \quad \|\eta\|_{\dvr}\leq c_5^{-1}\|v\|.
\end{equation}
Next, from \eqref{inf-sup-mod} and \cite[Lemma 3.1]{arnold2015mixed}
there exists $q \in Q_h$ such that
\begin{align}\label{q-eqn}
	P_{W_h}^{Q}\dvr S (q) = w - P_{W_h}^Q \asym \eta, \quad
\|q\|_1 \le c_6^{-1} \|w - P_{W_h}^Q \asym \eta\| \le C(\|w\| + \|v\|),
\end{align}
where $P_{W_h}^{Q}: L^2(\Omega) \to W_h$ is the $L^2$-projection with respect to 
the norm $(\cdot,\cdot)_Q$, satisfying, for $\varphi \in L^2(\Omega)$, 
{\color{black}$(P_{W_h}^{Q} \varphi - \varphi, w)_Q = 0 \,\, \forall w \in W_h$}. Now let
$$
\t = \eta - \curl q \in \X_h.
$$
Using \eqref{eta-eqn}, we have
\begin{align}\label{div-tau}
	(\dvr \t,v) = (\dvr \eta,v) = \|v\|^2,
\end{align}
and
\begin{equation}\label{tau-bound}
\|\tau\|_{\dvrg} \le C (\|\eta\|_{\dvrg} + \|q\|_1) \le C(\|w\| + \|v\|).
\end{equation}
Also, \eqref{curl-div} implies that $\asym \t = \asym \eta + \dvr S(q)$ 
and
\begin{align}
(\asym\t,w)_Q &= (\asym\eta, w)_Q + (\dvr S(q), w)_Q 
	    = (P_{W_h}^Q \asym\eta, w)_Q + (P_{W_h}^{Q}\dvr S(q),w)_Q \nonumber \\
	    &= (P_{W_h}^Q \asym\eta, w)_Q + (w - P_{W_h}^Q \asym \eta, w)_Q
= (w,w)_Q \ge C \|w\|^2. \label{asym-t}
\end{align}
Let $\xi = \Xi(w) \in \W_h^1$. Using \eqref{asym-identity}, \eqref{div-tau}, 
\eqref{asym-t}, and \eqref{tau-bound}, we obtain
\begin{align*}
	(\dvrg \tau,v) + (\tau, \xi)_Q = (\dvrg \tau,v) + (\asym \tau, w)_Q \geq c \|\tau\|_{\dvrg}(\|v\|+\|\xi\|),
\end{align*}
which completes the proof.
\end{proof}
{\color{black}
We proceed with the verification of the assumptions of Theorem \ref{msmfe-stability-1} for the spaces 
$\X_h\times V_h \times \W^1_h$ defined  in \eqref{spaces-with-j} and 
\eqref{mixed-spaces}. We first establish conditions \eqref{darcy-pair} 
and \eqref{curl-Q-S}. Condition \eqref{inf-sup-mod} is verified in the next section.
}
\begin{lemma}\label{aux-lemma}
Conditions \eqref{darcy-pair} and \eqref{curl-Q-S} hold for 
$\X_h\times V_h \times \W^1_h$ defined  in \eqref{spaces-with-j} and 
\eqref{mixed-spaces}.
\end{lemma}
\begin{proof}
{\color{black} The spaces $\X_h\times V_h \times \W^1_h$ defined  in \eqref{spaces-with-j} and 
\eqref{mixed-spaces} satisfy $\X_h = (S_h)^d$, $V_h = (U_h)^d$,
and $\W_h^1 = \Xi(W_h)$ with the spaces}
\begin{equation*}
S_h = \{ z\in H(\dvrg;\Omega): z|_E \overset{\Pc}\leftrightarrow \hat z \in \BDM_1(\Eh), \,
z\cdot n = 0 \mbox{ on } \Gamma_N\}, 
\end{equation*}
\begin{equation*}
U_h = \{r \in L^2(\Omega): r|_E \leftrightarrow \hat r \in \Pc_0(\Eh)\}, \quad
W_h = \{w \in H^1(\Omega): w|_E \leftrightarrow \hat w \in \Pc_1(\Eh)\}.
\end{equation*}
Note that, as shown Lemma~\ref{quad-inner-prod}, 
$W_h$ satisfies the norm equivalence $(w,w)_Q^{1/2} \sim \|w\|$.
The boundary condition in $S_h$ is needed to guarantee the essential
boundary condition in $\X_h$. Since $\BDM_1 \times \Pc_0$ is a stable
Darcy pair \cite{brezzi1991mixed}, \eqref{darcy-pair} holds. Next, we take
$$ 
Q_h = \{ q \in H^1(\Omega,\H) : q_i|_E \in \Pc_2,
\, i=1,\dots d^2(d-1)/2,\, q = 0 \mbox{ on } \Gamma_N \}. 
$$
Note that $\curl \Pc_2(\H) \subset (\BDM_1)^{d}$. The boundary condition in
$Q_h$ guarantees that $\curl Q_h \subset (S_h)^d$, i.e., \eqref{curl-Q-S} holds. In particular,
$(\curl q) n = 0 \mbox{ on } \Gamma_N \, \forall q\in Q_h$, which follows
from the following lemma.
\end{proof}
\begin{lemma}
Let $\Omega$ be a bounded domain of $\R^d,\, d=2,3$ and let $\varphi
\in H^1(\Omega,\R^{d(d-1)/2})$ such that $\varphi=0 \mbox{ on } \Gamma$, where
$\Gamma$ is a non-empty part of $\partial\Omega$. Then
$(\curl \varphi)\cdot n = 0$ on $\Gamma$.
\end{lemma}
\begin{proof}
In 2D, let $t = (t_1, t_2)^T$ be the unit tangential vector on $\Gamma$. 
The assertion of the lemma follows from
$$   
0 = \nabla \varphi\cdot t 
= (\partial_x \varphi) t_1 + (\partial_y \varphi) t_2 
= (\partial_x \varphi) n_2 - (\partial_y \varphi) n_1 
= - \curl \varphi \cdot n.
$$
{\color{black}
In 3D, let $\varphi = (\varphi_1,\varphi_2,\varphi_3)^T$, and $n =
(n_1,n_2,n_3)^T$. Since $\varphi = 0$ on $\Gamma$, it holds that
$\nabla\varphi_i\cdot t = 0$ on $\Gamma$, $i = 1,2,3$, for any
tangential vector $t$. We have
\begin{align*}
& 0 = \nabla \varphi_1\cdot (0,-n_3,n_2)^T = - (\partial_y\varphi_1)n_3 + (\partial_z \varphi_1)n_2,\\
& 0 = \nabla \varphi_2\cdot (n_3,0,-n_1)^T = (\partial_x\varphi_2)n_3 - (\partial_z \varphi_2)n_1,\\
& 0 = \nabla \varphi_3\cdot (-n_2,n_1,0)^T = - (\partial_x\varphi_3)n_2 + (\partial_y \varphi_3)n_1,
\end{align*}
which implies that
\begin{align*}
(\curl \varphi)\cdot n = (\partial_y \varphi_3 - \partial_z\varphi_2)n_1 
+ (\partial_z \varphi_1 - \partial_x\varphi_3)n_2 
+ (\partial_x \varphi_2 - \partial_y\varphi_1)n_3 = 0.
\end{align*}
}
\end{proof}
To show \ref{S4}, it remains to show that \eqref{inf-sup-mod} holds.
It is well known that $\Pc_2 - \Pc_1$ is a stable Taylor-Hood pair for
the Stokes problem \cite{brezzi1991mixed}. However, this does not imply
the inf-sup condition with quadrature \eqref{inf-sup-mod}. We show that
it holds in the next sections.

\subsubsection{The inf-sup condition for the Stokes problem} 
In the following, for simplicity, we let $b(q,w) = (\dvr q,w)$ and 
$b(q,w)_Q = (\dvr q,w)_Q$. We will show the inf-sup condition \eqref{inf-sup-mod}
for spaces $Q_h \subset H^1(\Omega,\R^d)$ and $W_h \subset L^2(\Omega)$, which 
will imply the statement for 
$Q_h \subset H^1(\Omega,\H)$ and $W_h \in L^2(\Omega,\R^{d(d-1)/2})$.
Adopting the approach by Stenberg \cite{stenberg1984analysis} we
introduce a macroelement condition that is sufficient for
\eqref{inf-sup-mod} to hold. A macroelement is a union of one
or more neighboring simplices, satisfying the usual shape-regularity
and connectivity conditions. We say that a macroelement $M$ is
equivalent to a reference macroelement $\hat{M}$, if there is a mapping
$F_M : \hat{M} \to M$, such that
 \begin{enumerate} [label=(\roman*)]
    \item $F_M$ is continuous and one-to-one;
    \item $F_M(\Mh) = M$;
    \item If $\Mh = \cup_{j=1}^m \hat{T}_j$, then $M = \cup_{j=1}^m T_j$
where $T_j = F_M(\hat{T}_j),\, j=1,\dots,m$;
    \item $F_{M|_{\hat{T}_j}}= F_{T_j}\circ F_{\hat{T}_j}^{-1},\, j=1,\dots,m,$ where $F_{\hat{T}_j}$ and $F_{T_j}$ are the affine mappings from the reference simplex onto $\hat{T}_j$ and $T_j$, respectively.
 \end{enumerate} 

The family of macroelements equivalent to $\hat{M}$ is denoted by 
$\mathcal{E}_{\hat{M}}$. Let
\begin{align*}
    Q^0_{M} =\{ q\in H_0^1(M,\R^d): q_{i}|_T\in \mathcal{P}_2,\, 
i=1,\dots,d,\, \forall T\subset M\},
\quad
    W_{M} &= \{ w \in H^1(M) : w|_T \in \mathcal{P}_1,\, \forall T\subset M \},
\end{align*}
\begin{align*}
    W^0_{M} = W_M \cap L^2_0(M), \quad
    N_M =\{ w\in W_M: b(q,w)_{Q,M}=0,\, \forall q \in Q^0_{M} \}.
\end{align*}
We assume that there is a fixed set of classes 
$\mathcal{E}_{\hat{M}_i},\, i=1,...,n$ such that
 \begin{enumerate} [label={\bf(M\arabic*)},align=left]
    \item \label{M1} For each $M \in \mathcal{E}_{\hat{M}_i}$, the space $N_M$ is one-dimensional, 
consisting of constant functions;
    \item \label{M2} There exists a partition $\Mc_h$ of $\mathcal{T}_h$ into macroelements 
$M \in \mathcal{E}_{\hat{M}_i}, \, i=1,...,n$.
 \end{enumerate} 

\begin{theorem}\label{macro-inf-sup}
If \ref{M1}--\ref{M2} are satisfied, then the Stokes inf-sup condition with 
quadrature \eqref{inf-sup-mod} holds.
\end{theorem}
Before we prove this result, we prove three auxiliary lemmas, following the 
argument in \cite{stenberg1984analysis}. 

\begin{lemma}\label{macro-lemma-1}
If \ref{M1} holds, then there exists a constant $\beta > 0$ independent of $h$ 
such that, 
\begin{align*}
\forall \, M \in \mathcal{E}_{\hat{M}_i}, \quad
\sup_{0\neq q \in Q^0_{M}}\frac{b(q,w)_{Q,M} }{|q|_{1,M}} 
\geq \beta\|w\|_{M},\, \forall w \in W^0_{M}.
\end{align*}
\end{lemma}
\begin{proof}
The assertion of the lemma follows from \ref{M1} and a scaling argument, see
\cite[Lemma 3.1]{stenberg1984analysis}.
\end{proof}

Next, let $\mathbb{P}_h$ denote the $L^2$-projection from $W_h$ onto the space
\begin{align*}
M_h =\{\mu \in L^2(\Omega): \mu\big|_M \text{ is constant }\forall \,
M \in \mathcal{M}_h\}.
\end{align*}

\begin{lemma}\label{macro-lemma-2}
If \ref{M1}--\ref{M2} hold, then there exists a constant $C_1 >0$, such that 
for every $w \in W_h$, there exists $q \in Q_h$ satisfying
\begin{align*}
b(q,w)_{Q} = b(q, (I-\mathbb{P}_h)w)_{Q} \geq C_1 \|(I-\mathbb{P}_h)w\|^2, 
\quad \mbox{and} \quad |q|_1 \leq \|(I-\mathbb{P}_h)w\|.
\end{align*}
\end{lemma}
\begin{proof}
For every $w \in W_h$ we have
$(I -\mathbb{P}_h)w \in W^0_{M},\, \forall M \in \mathcal{M}_h$. 
Then Lemma \ref{macro-lemma-1} implies that for every $M$ there exists 
$q_M \in Q^0_{M}$ such that
\begin{align*}
b(q_M,(I-\mathbb{P}_h)w)_{Q,M} \geq C \|(I-\mathbb{P}_h)w\|^2_{M} 
\quad \mbox{and} \quad |q_M|_{1,M} \leq \|(I-\mathbb{P}_h)w\|_{M},
\end{align*}
Define $q \in Q_h$ by $q \big|_M=q_M, \,\, \forall M \in \mathcal{M}_h$.
It follows from \ref{M1} that
$b(q,\mathbb{P}_h w)_{Q} =0, \,\, \forall w \in W_h$.
{\color{black}Then we have, }
\begin{align*}
b(q,w)_{Q} = b(q,(I-\mathbb{P}_h)w)_{Q}
=\sum_{M\in \mathcal{M}_h}b(q_M,(I-\mathbb{P}_h)w)_{Q,M} 
\geq C \|(I-\mathbb{P}_h)w\|^2,
\end{align*}
which completes the proof.
\end{proof}
\begin{lemma}\label{macro-lemma-3}
There exists a constant $C_2>0$ such that for every 
$w\in W_h$ there exists $g \in Q_h$ such that
\begin{align*}
    b(g,\mathbb{P}_h w)_{Q} =\|\mathbb{P}_h w\|^2 \quad \text{ and }\quad 
\|g\|_1 \leq C_2\|\mathbb{P}_h w\|.
\end{align*}
\end{lemma}
\begin{proof}
Let $w\in W_h$ be arbitrary. There exists $z \in H^1(\Omega)$, $z = 0$ on
$\Gamma_N$, such that
\begin{align*}
    \dvr z = \mathbb{P}_h w \quad \mbox{and}\quad \|z\|_1 \leq C\|\mathbb{P}_h w\|.
\end{align*}
This follows from \cite{Galdi} by choosing $z = \varphi$ on $\Gamma_D$, where
$\varphi$ is a smooth function with compact support on $\Gamma_D$ such that
$\int_{\Gamma_D} \varphi\cdot n = \int_{\Omega} \mathbb{P}_h w$. We next consider
an operator $I_h:H^1(\Omega) \rightarrow Q_h$ such that
\begin{align}\label{div-Ih}
(\dvr I_h z, \mu) = (\dvr z,\mu) \,\, \forall \mu \in M_h, \quad 
\|I_h z\|_1 \leq C\|z\|_1.
\end{align}
Such an operator is constructed in \cite[Lemma 3.5]{stenberg1984analysis}, by setting the velocity 
degrees of freedom at the midpoints of facets $e$ on the interfaces between macroelements such that
$\int_e I_h z = \int_e z$, which guarantees \eqref{div-Ih}, and local averages for the rest of the 
degrees of freedom. Finally, since the vertex quadrature rule is exact for linear functions, 
we have that $(\dvr I_h z, \mu)_{Q} = (\dvr I_h z, \mu)$, so we can take $g = I_h z$.
\end{proof}
We are now ready to prove the main result stated in Theorem \ref{macro-inf-sup}:
 \begin{proof}[Proof of Theorem \ref{macro-inf-sup}] 
Let $w \in W_h$ be given, and let $q \in Q_h$ and $g\in Q_h$ be the functions
constructed in Lemma \ref{macro-lemma-2} and Lemma \ref{macro-lemma-3}, respectively.
Set $z = q + \delta g$, where $\delta = 2C_1(1+C_2^2)^{-1}$. We then have
\begin{align*}
    b(z,w)_{Q} &=b(q,w)_{Q}+\delta b(g,w)_{Q}  = b(q,w)_{Q}+\delta b(g,\mathbb{P}_h w)_{Q} 
+ \delta b(g,(I-\mathbb{P}_h) w)_{Q} \\
        &\geq C_1\|(I-\mathbb{P}_h)w\|^2 
+\delta \|\mathbb{P}_h w\|^2 - \delta |g|_1\|(I-\mathbb{P}_h)w\| \\
        & \geq C_1(1+C_2^2)^{-1}\|w\|^2,
\end{align*}
and $\|z\|_1 \leq \|(I-\mathbb{P}_h)w\| +\delta C_2\|\mathbb{P}_h w\| \leq C \|w\| $, 
implying that \eqref{inf-sup-mod} holds.
\end{proof}

\subsubsection{Verification of macroelement condition \ref{M1}}
We consider macroelements of the following type. 
\begin{definition}\label{macro-defn}
Each macroelement $M$ is associated with an interior vertex $c$ in
$\mathcal{T}_h$, consisting of all simplices that share that vertex.
\end{definition}
We note that $c$ is the only interior vertex of $M$. All other
vertices are on $\partial M$ and each vertex is connected to $c$ by an
edge. A 2D example of a macroelement that satisfies Definition
\ref{macro-defn} is shown on Figure \ref{fig:macro_patch}. We next
show that \ref{M1} holds.

\begin{figure}
\centering
\begin{minipage}[b]{0.47\textwidth}
\begin{center}
\scalebox{.8}{
\begin{tikzpicture}
    [inner sep=0.65mm,
    displ/.style={circle,draw=black,fill=black,thick},
    rotat/.style={circle,draw=black,fill=white,very thick}]
    \coordinate (p1) at (0.5,0.3);
    \coordinate (p2) at (3.5,0.5);
    \coordinate (p3) at (4.5,1.9); 
    \coordinate (p4) at (3.5,3.9);
    \coordinate (p5) at (0.7,4.1);
    \coordinate (p6) at (0.0,2.1);
    \coordinate (p7) at (2.0,2.0);

    \draw [very thick] (p1) -- (p2);
    \draw [very thick] (p2) -- (p3);
    \draw [snake it] (p3) -- (p4);
    \draw [very thick] (p4) -- (p5);
    \draw [very thick] (p5) -- (p6);
    \draw [very thick] (p6) -- (p1);    
    \draw [very thick] (p7) -- (p1);
    \draw [very thick] (p7) -- (p2);
    \draw [very thick] (p7) -- (p3);
    \draw [very thick] (p7) -- (p4);
    \draw [very thick] (p7) -- (p5);
    \draw [very thick] (p7) -- (p6);

    \node at (p1) [displ] {};
    \node at (p2) [displ] {};
    \node at (p3) [displ] {};
    \node at (p4) [displ] {};
    \node at (p5) [displ] {};
    \node at (p6) [displ] {};
    \node at (p7) [displ] {};

    \draw (0.5,0.3) node[below=3pt] {\small$\r_4$};
    \draw (3.8,0.5) node[below=3pt] {\small$\r_5$};
    \draw (4.5,1.9) node[right=3pt] {\small$\r_6$};
    \draw (3.6,4.0) node[right=3pt] {\small$\r_{N+1}$};
    \draw (0.7,4.1) node[left=3pt] {\small$\r_2$};
    \draw (0.0,2.1) node[left=3pt] {\small$\r_3$};
    \draw (2.1,2.0) node[below=4pt] {\small$\r_1$};

    \draw (0.5,2.7) node[below=3pt] {\small$T_1$};
    \draw (0.6,2.0) node[below=3pt] {\small$T_2$};
    \draw (0.8,0.6) node[right=3pt] {\small$T_3$};
    \draw (3.45,0.7) node[above=3pt] {\small$T_4$};
    \draw (1.4,4.1) node[below=3pt] {\small$T_{N}$};
    \draw (3.4,2.4) node[above=3pt] {$\dots$};

\end{tikzpicture}
}
\end{center}
\caption{Macroelement with $N$ triangles \newline}
\label{fig:macro_patch}
\end{minipage}
\hspace{2em}
\begin{minipage}[b]{0.47\textwidth}
\begin{center}
\scalebox{.8}{
\hspace{-10em}
\begin{tikzpicture}
    [inner sep=0.65mm,
    displ/.style={circle,draw=black,fill=black,thick},
    rotat/.style={circle,draw=black,fill=white,very thick},
    veloc/.style={diamond,draw=black,fill=black,thick}]

    \coordinate (q1) at (1.6,1.4);
    \coordinate (q2) at (4.8,1.8);
    \coordinate (q3) at (1.9,4.0); 
    \coordinate (q4) at (4.5,4.3);
    \coordinate (dof2) at (3.27,2.95);

    \draw [very thick] (q1) -- (q2);
    \draw [very thick] (q3) -- (q4);
    \draw [very thick] (q1) -- (q3);
    \draw [very thick] (q2) -- (q4);
    \draw [thick] (q3) -- (q2);

    \node at (q2) [displ] {};
    \node at (q3) [displ] {};

    \node at (q1) [rotat] {};
    \node at (q4) [rotat] {};

    \node at (dof2) [veloc] {};    

    \draw (1.6,1.4) node[below=3pt] {\small$\r_1$};
    \draw (4.9,1.8) node[below=3pt] {\small$\r_2$};
    \draw (4.6,4.3) node[above=3pt] {\small$\r_3$};
    \draw (1.8,4.0) node[above=3pt] {\small$\r_4$};
    
    \draw (2.0,2.2) node[below=3pt] {\small$T_1$};
    \draw (4.2,4.1) node[below=3pt] {\small$T_2$};

\end{tikzpicture}
}
\end{center}
\caption{Union of two triangles; \\
$\Pc_2-\Pc_1$ degrees of freedom.}
\label{fig:macro_reference}
\end{minipage}

\end{figure}

\begin{lemma}\label{macro-lemma-constantnullspace}
The macroelements $M$ described in Definition~\ref{macro-defn} satisfy \ref{M1}.
\end{lemma}
\begin{proof}
For the sake of space, we present the proof for the 2D case. The
extension to 3D is straightforward. We first consider a union of two
triangles, $T_1 \cup T_2$, sharing an edge, as shown on Figure
\ref{fig:macro_reference}, and compute
\begin{align*}
(\dvr q_j, w)_{Q,T_1 \cup T_2} = \sum_{i=1}^2 
\left(\tr\left(\nabla q_j\right), w\right)_{Q,T_i} 
= \sum_{i=1}^2 \left(\tr \left(DF_{T_i}^{-T}\hat{\nabla} \hat{q}_j\right), 
\hat{w}J_{T_i}\right)_{\hat{Q},\hat{T}}, \quad j=1,2,
\end{align*}
where $q_1$ and $q_2$ are the velocity degrees of freedom associated
with the midpoint of edge $\r_{24}$.  Let us assume
that $F_{T_1}: \hat T \rightarrow T_1$ maps $\rh_1 \rightarrow \r_1,\, \rh_2
\rightarrow \r_2$ and $\rh_3 \rightarrow \r_4$. Then $DF_{T_1} = [\r_{21},\r_{41}]$ and 
we have

\begin{align*}
\hat{q}_1 = \begin{pmatrix} 4\hat{x}\hat{y} \\ 
0 \end{pmatrix}, \quad \hat{q}_2 = \begin{pmatrix} 0 
\\ 4\hat{x}\hat{y} \end{pmatrix}, \quad 
DF^{-T}_{T_1} = \frac{1}{J_{T_1}}\begin{pmatrix}
 y_4-y_1 & x_1-x_4 \\
 y_1-y_2 & x_2-x_1
 \end{pmatrix},
\end{align*}
which implies
\begin{align*}
(\dvr q_1, w)_{Q,T_1} &= \frac{2}{3}\left( (y_1-y_2)w(\r_2) + (y_4-y_1)w(\r_4) \right),\\
(\dvr q_2, w)_{Q,T_1} &= \frac{2}{3}\left( (x_2-x_1)w(\r_2) + (x_1-x_4)w(\r_4) \right). 
\end{align*}
Similarly, let $F_{T_2}: \hat T \rightarrow T_2$ map
$\rh_1 \rightarrow \r_2,\, \rh_2 \rightarrow \r_3$ and
$\rh_3 \rightarrow \r_4$. Then we have
\begin{align*}
\hat{q}_1\big|_{\hat{T}} = \begin{pmatrix} 4\hat{y} - 4\hat{x}\hat{y}-4\hat{y}^2 
\\ 0 \end{pmatrix}, \quad \hat{q}_2\big|_{\hat{T}} = \begin{pmatrix} 0 
\\ 4\hat{y} - 4\hat{x}\hat{y}-4\hat{y}^2 \end{pmatrix}, \quad
DF^{-T}_{T_2} = \frac{1}{J_{T_2}}\begin{pmatrix}
 y_4-y_2 & x_2-x_4 \\
 y_2-y_3 & x_3-x_2
 \end{pmatrix},
\end{align*}
which implies
\begin{align*}
(\dvr q_1, w)_{Q,T_2} &= \frac{2}{3}\left( (y_2-y_3)w(\r_2) + (y_3-y_4)w(\r_4) \right),\\
(\dvr q_2, w)_{Q,T_2} &= \frac{2}{3}\left( (x_3-x_2)w(\r_2) + (x_4-x_3)w(\r_4) \right). 
\end{align*}
Therefore, we obtain
\begin{align*}
(\dvr q_1, w)_{Q,T_1 \cup T_2} &= \frac{2}{3} (y_1-y_3)(w(\r_2) - w(\r_4)),\\
(\dvr q_2, w)_{Q,T_1 \cup T_2} &= \frac{2}{3} (x_3-x_1)(w(\r_2) - w(\r_4)). 
\end{align*}
Since $x_1 - x_3$ and $y_1 - y_3$ cannot be both zero, it follows from 
$(\dvr q_1, w)_{Q,T_1 \cup T_2}=0$ and $(\dvr q_2, w)_{Q,T_1 \cup T_2}=0$ that $w(\r_2)=w(\r_4)$. 

Let $M$ be a macroelement described in Definition~\ref{macro-defn} and let $w \in N_M$.  
The above argument can
be applied to every pair of triangles in $M$ that
share an edge, which implies that for every interior edge, the values
of $w$ at the interior vertex and the boundary vertex are equal. Since
all boundary vertices are connected to the interior vertex, this 
implies that $w$ has the same value at all vertices, i.e., $w$ is
a constant on $M$.
On the other hand, if $w$ is a constant on $M$, since the 
quadrature rule is exact for linear functions on each $T_i$, we have 
for any $q \in Q_M^0$,
$$
(\dvr q, w)_{Q,M} = \sum_{i=1}^N (\dvr q, w)_{Q,T_i} =  \sum_{i=1}^N (\dvr q, w)_{T_i} = (\dvr q, w)_M = 
- (q, \nabla w)_M = 0.
$$
Therefore, $N_M$ is one-dimensional, consisting of constant functions.
\end{proof}

We are now ready to prove the well-posedness of the MSMFE-1 method.
\begin{theorem}\label{MSMFE1-well-posed}
Assuming that \ref{M2} holds with macroelements described in Definition~\ref{macro-defn}, 
then he MSMFE-1 method \eqref{msmfe-1-1}--\eqref{msmfe-1-3} has a unique solution.
\end{theorem}

\begin{proof}
The existence and uniqueness of a solution to \eqref{msmfe-1-1}--\eqref{msmfe-1-3} follows from
\ref{S3} and \ref{S4}. Lemma~\ref{quad-inner-prod} implies the coercivity condition \ref{S3}. 
Assuming \ref{M2}, the inf-sup condition \ref{S4} follows from a combination of 
Theorem~\ref{msmfe-stability-1}, Lemma~\ref{aux-lemma}, Theorem~\ref{macro-inf-sup}, and
Lemma~\ref{macro-lemma-constantnullspace}.
\end{proof}


\subsection{Reduction to a cell-centered displacement system of the MSMFE-1 method}
We recall the displacement-rotation system \eqref{msmfe0-system} of the MSMFE-0 method, obtained 
after a local stress elimination. In the MSMFE-1 method, the matrix $\Ag$ is different from the 
MSMFE-0 method, since it involves the quadrature rule, i.e., $(\Ag)_{ij} = (\tau_j,\xi_i)_Q$.
Since the quadrature rule localizes the interaction of basis functions around
each vertex, $(\Ag)$ is block-diagonal with $d(d-1)/2 \times dk$ blocks with elements
$(\Ag)_{ij} = (\tau_j,\xi_i)_Q$, $i = 1,\dots,d(d-1)/2$, $j = 1,\dots,dk$.

\begin{lemma}\label{C-diagonal}
The matrix $\Ag\As^{-1}\Ag^T$ in the MSMFE-1 method is block-diagonal and invertible.
\end{lemma}
\begin{proof}
Since $(\Ag)$ is block-diagonal with $d(d-1)/2 \times dk$ blocks and
$\As$ is block-diagonal with $dk\times dk$ blocks, then
$\Ag\As^{-1}\Ag^T$ is block-diagonal with $d(d-1)/2 \times d(d-1)/2$
blocks. Note that for $d=2$ the blocks are $1\times 1$, i.e., the
matrix is diagonal, and for $d = 3$ the blocks are $3\times 3$. The
blocks couple the $d(d-1)/2$ rotation degrees of freedom associated
with a vertex. Each block is invertible, due to the inf-sup condition
\ref{S4} and the fact that the blocks of $\As^{-1}$ are symmetric and
positive definite.
\end{proof}
The above result implies that the rotation $\g$ can be easily eliminated from the system
\eqref{msmfe0-system} by solving local $d(d-1)/2 \times d(d-1)/2$ problems, resulting
in a cell-centered system for the displacement $u$:
\begin{equation} \label{msmfe1-system}
    (\Au\As^{-1}\Au^T - \Au\As^{-1}\Ag^T (\Ag\As^{-1}\Ag^T)^{-1} \Ag\As^{-1}\Au^T)u = \hat{f}.
\end{equation}
\begin{lemma}
The matrix in \eqref{msmfe1-system} is symmetric and positive definite.
\end{lemma}
\begin{proof}
The matrix \eqref{msmfe1-system} is a Schur complement of the matrix
in \eqref{msmfe0-system}, which is symmetric and positive definite due
to the inf-sup condition \ref{S4} and the proof of
Lemma~\ref{disp-rot-spd}. A well known result from linear
algebra \cite[Theorem 7.7.6]{Horn-Johnson} implies that the 
matrix \eqref{msmfe1-system} is also symmetric and positive
definite.
\end{proof}


\section{Error analysis}
In this section we analyze the convergence of
the proposed methods. We will use several well
known projection operators. We consider the $L^2$-orthogonal
projection $\Quh: V \to V_h$ such that 
\begin{equation} \label{proj-prop-1}
(v-\Quh v, w) = 0, \qquad \forall w\in V_h,
\end{equation}
and the $L^2$-orthogonal projection $\Qgh: \W \to \W_h^k$, $k = 0,\,1$ such that
\begin{equation} \label{proj-prop-2}
(\xi - \Qgh \xi, \zeta) = 0, \qquad \forall \zeta \in \W_h^k, \,\, k=0,1.
\end{equation}
We also consider the MFE projection operator \cite{brezzi1985two,brezzi1991mixed}
$\Pi: \, \X \cap H^1(\Omega,\M) \to \X_h$ such that
\begin{equation} \label{space-div-prop}
(\dvr (\Pi\tau - \tau), v) = 0, \qquad \forall v \in V_h.
\end{equation}
These operators have approximation properties \cite{ciarlet2002finite,brezzi1985two,brezzi1991mixed}
\begin{align}
    &\| v - \Quh v \| \le C h^r \|v\|_r,          && 0\le r\le 1, \label{prop-Quh}\\
    &\| \xi - \Qgh \xi \| \le C h^r \|\xi\|_r, && 0\le r\le 1, \label{prop-Qgh}\\
  &\| \tau - \Pi\tau \| \le C h^r \|\tau\|_r,   && 1\le r\le 2, \label{prop-Pi}\\
  & \|\dvr(\tau - \Pi\tau)\| \le C h^r \|\dvr \tau\|_r, && 0 \le r \le 1. \label{prop-div}
\end{align}
For $\varphi \in L^2(E)$, let $\bar\varphi$ be its mean value on $E$, which satisfies
\begin{equation}\label{mean-value}
\|\varphi - \bar\varphi\|_E \le C h \|\varphi\|_{1,E}, \quad
\|\varphi - \bar\varphi\|_{\infty,E} \le C h \|\varphi\|_{\infty,E}.
\end{equation}
We will also use the inverse inequality for a finite element function $\varphi$ \cite{ciarlet2002finite}
\begin{equation}\label{inverse}
  \|\varphi\|_{j,E} \le C h^{-1} \|\varphi\|_{j-1,E}, \quad j \ge 1.
  \end{equation}
We will make use of the following continuity bounds. 
\begin{lemma} \label{proj-continuity}
For all elements $E$ there exist a constant $C$ independent of $h$ such that
    \begin{align}
        &\|\Pi\t\|_{1,E} \le C\|\t\|_{1,E}, \quad \forall\t\in H^{1}(E,\M), \label{mixed-proj-continuity} \\
        &\|\Qgh \xi\|_{1,E} \le C\|\xi\|_{1,E}, \quad \forall \xi\in H^1(E, \N). \label{continuity-l2-proj}
    \end{align}
\end{lemma}
\begin{proof}
To prove \eqref{mixed-proj-continuity} we write
$$
 |\Pi \t|_{1,E} = |\Pi\t - \bar\t|_{1,E} \le C h^{-1} \|\Pi\t - \bar\t\|_E
 \le C h^{-1}(\|\Pi\t - \t\|_E + \|\t - \bar\t\|_E) \le C \|\t\|_{1,E},
$$
where we have used \eqref{inverse}, \eqref{prop-Pi}, and \eqref{mean-value}. The above inequality, combined with
$\|\Pi\t\|_E \le C \|\t\|_{1,E}$, which follows from \eqref{prop-Pi}, implies \eqref{mixed-proj-continuity}.
The proof of \eqref{continuity-l2-proj} is similar.
\end{proof}
We next derive bounds for quadrature error. We will use the notation $A \in W^{j,\infty}_{\Tc_h}$ if
$A \in W^{j,\infty}(E) \, \forall E \in \Tc_h$ and $\|A\|_{j,\infty,E}$ is uniformly bounded independently of $h$.
\begin{lemma} \label{q-err-h1-lem}
  If $A \in W^{1,\infty}_{\Tc_h}$, there exists a constant $C$ independent of $h$ such that for all $\tau,\chi \in \X_h$,
  $\xi \in \W^1_h$,
\begin{align}
&| \theta(A\chi,\tau) | \le C \sum_{E\in\Tc_h} h\|A\|_{1,\infty,E}\|\chi\|_{1,E}\|\tau\|_{E}, \label{q-err-h1-eq1}\\
&| \del(\t,\xi) | \le C \sum_{E\in\Tc_h} h\|\t\|_{E}\|\xi\|_{1,E} \label{q-err-h1-eq2}, \\
&| \del(\t,\xi) | \le C \sum_{E\in\Tc_h} h\|\t\|_{1,E}\|\xi\|_{E} \label{q-err-h1-eq3}.
\end{align}
\end{lemma}
\begin{proof}
  For \eqref{q-err-h1-eq1} we write on any element $E$, using Lemma~\ref{0-q-err-const}, Lemma~\ref{quad-inner-prod},
  and \eqref{mean-value},
\begin{align*}
     | \tet_E(A\chi,\t) | \le | \tet_E\inp[(A-\bar{A})\chi]{\t} | + | \tet_E\inp[\bar{A}(\chi - \bar\chi)]{\t}| 
\le Ch(|A|_{1,\infty,E}\|\chi\|_E \|\t\|_E + \|A\|_{0,\infty,E}\|\chi\|_{1,E}\|\t\|_{E}).
\end{align*}
Similarly, using Lemma~\ref{0-q-err-const}, Lemma~\ref{quad-inner-prod},
  and \eqref{mean-value}, we have 
\begin{align*}
     | \del_E\inp[\t]{\xi} | = | \del_E\inp[\t]{\xi - \bar{\xi}} | 
\le Ch \|\t\|_E\|\xi\|_{1,E} 
     \quad\mbox{and} \quad
        | \del_E\inp[\t]{\xi} | = | \del_E\inp[\t - \bar{\t}]{\xi} | 
\le Ch \|\t\|_{1,E}\|\xi\|_{E} .
\end{align*}
The proof is completed by summing over the elements.
\end{proof}



\subsection{First order convergence for all variables}
\begin{theorem}
Let $A\in W^{1,\infty}_{\Tc_h}$.
For the solution $(\s,u,\g)$ of \eqref{weak-elast} and its numerical approximation
$(\s_h,u_h,\g_h)$ obtained by either the MSMFE-0 method 
\eqref{msmfe-0-1}--\eqref{msmfe-0-3} or the MSMFE-1 method 
\eqref{msmfe-1-1}--\eqref{msmfe-1-3}, there exists a
constant $C$ independent of $h$ such that
\begin{align} \label{error-estimate}
\|\sigma-\sigma_h\|_{\dvrg} + \| u-u_h\|+\| \g-\g_h\|  \leq Ch(\|\sigma\|_1 + \|\dvr \sigma\|_1 + \|u\|_1+\|\g\|_1).
\end{align}
\end{theorem}
\begin{proof}
We present the argument for the MSMFE-1 method, which includes the
proof for the MSMFE-0 method, as noted below. Subtracting the
numerical method \eqref{msmfe-1-1}-\eqref{msmfe-1-3} from the
variational formulation \eqref{weak-elast}, we obtain the error
equations
\begin{align}
(A\sigma,\t) - (A\s_h,\t)_Q + (u-u_h,\dvr\tau) + (\g, \tau) - (\g_h, \tau)_Q&= 0, \quad \tau \in \X_h, \label{err-m1-1}\\
(\dvr(\sigma - \sigma_h), v) &= 0, \quad v \in V_h, \label{err-m1-2}\\
(\sigma, \xi) - (\sigma_h, \xi)_Q &= 0, \quad \xi \in \W_h^1. \label{err-m1-3}
\end{align}
Using \eqref{space-div-prop}, \eqref{quad-err-def}, \eqref{proj-prop-1}, and that $\dvr \X_h = V_h$,
we can rewrite the above error system as 
\begin{align}
&(A(\Pi\sigma - \sigma_h),\tau)_Q  + (\Quh u-u_h,\dvr\tau)+ (\tau, \Qgh \g-\g_h)_Q \nonumber\\
& \qquad\quad = \inp[A(\Pi\s-\s)]{\t} - \tet\inp[A\Pi\s]{\t} 
+ \inp[\t]{\Qgh\g - \g} - \del\inp[\t]{\Qgh\g}, \label{eq1-err-msmfe1} \\
&    \dvr(\Pi\sigma - \sigma_h) = 0 \label{eq2-err-msmfe1}, \\
&    (\Pi\sigma-\sigma_h, \xi)_Q = \inp[\Pi\s-\s]{\xi} - \del\inp[\Pi\s]{\xi}. \label{eq3-err-msmfe1}
\end{align}
We proceed by giving bounds for the terms on the right in 
\eqref{eq1-err-msmfe1} and \eqref{eq3-err-msmfe1}, using
Cauchy-Schwarz and Young's inequalities. Bound \eqref{prop-Pi} yields
\begin{align} \label{bound1-msmfe0}
\inp[A(\Pi\s - \s)]{\t} + \inp[\Pi\s-\s]{\xi} 
\le Ch\|\s\|_1 (\|\t\| + \|\xi\|) \le Ch^2\|\s\|^2_1 +  \epsilon\|\t\|^2 +  \epsilon\|\xi\|^2.
\end{align}
It follows from \eqref{q-err-h1-eq1} and \eqref{mixed-proj-continuity} that
\begin{align} \label{bound2-msmfe0}
|\tet\inp[A\Pi\s]{\t}| \le C\sum_{E\in \Tc_h} h\|A\|_{1,\infty,E}\|\Pi\s\|_{1,E}\|\t\|_E \le Ch\|A\|_{1,\infty}\|\s\|_1\|\t\| \le Ch^2\|\s\|_1^2 + \epsilon\|\t\|^2.
\end{align}
It follows from \eqref{prop-Qgh} and \eqref{prop-Pi} that
\begin{align}
&\inp[\t]{\Qgh \g - \g} \le Ch\|\t\| \|\g\|_1 \le Ch^2\|\g\|^2_1 + \epsilon\|\t\|^2 \label{bound6-msmfe1}.
\end{align}
Using \eqref{q-err-h1-eq2}--\eqref{q-err-h1-eq3} and \eqref{mixed-proj-continuity}--\eqref{continuity-l2-proj}, 
we obtain
\begin{align} 
    &|\del\inp[\t]{\Qgh\g}| \le C\sum_{E\in\Tc_h} h\|\Qgh\g\|_{1,E}\|\t\|_E \le Ch\|\g\|_1\|\t\| \le Ch^2\|\g\|^2_1 + \epsilon\|\t\|^2, \label{bound5-msmfe1}\\
	&|\del\inp[\Pi\s]{\xi}| \le C\sum_{E\in\Tc_h}h\|\Pi\s\|_{1,E} \|\xi\|_E \le Ch\|\s\|_1 \|\xi\| \le Ch^2\|\s\|^2_1 +  \epsilon\|\xi\|^2 \label{bound7-msmfe1}.
\end{align}
Now, choosing $\t = \Pi\s - \s_h$ and $\xi = \Qgh \g - \g_h$ in \eqref{eq1-err-msmfe1} and \eqref{eq3-err-msmfe1},
and using \eqref{eq2-err-msmfe1}, gives 
\begin{align*}
    &(A(\Pi\sigma - \sigma_h),\Pi\s - \s_h)_Q =\inp[A(\Pi\s-\s)]{\Pi\s - \s_h} - \tet\inp[A\Pi\s]{\Pi\s - \s_h} \\
    &\qquad + \inp[\Pi\s - \s_h]{\Qgh\g-\g} - \del\inp[\Pi\s-\s_h]{\Qgh\g} 
    - \inp[\Pi\s-\s]{\Qgh \g - \g_h} + \del\inp[\Pi\s]{\Qgh \g - \g_h}.
\end{align*}
Combining \eqref{bound1-msmfe0}--\eqref{bound7-msmfe1}, using \eqref{norm-equiv}, and choosing 
$\epsilon$ small enough, we obtain
\begin{align} \label{stress-bound-msmfe1}
    \| \Pi \s - \s_h \|^2 \le Ch^2( \| \s \|_1^2 + \| \g \|_1^2) + \epsilon \| \Qgh \g - \g_h\|^2.
\end{align}
Using the inf-sup condition \ref{S4}, we have
\begin{align*}
&\|\Quh u-u_h\|+\|\Qgh \g-\g_h\| \nonumber \\ 
& \qquad \leq C \sup_{\tau \in \X_h}\frac{1}{\|\tau\|_{\dvrg}}
\big( \inp[A(\Pi\s-\s)]{\t} - \inp[A(\Pi\s-\s_h)]{\t}_Q - \tet\inp[A\Pi\s]{\t} - \del\inp[\t]{\Qgh\g} \big) 
\nonumber\\
& \qquad  \leq C\left( \|\Pi\sigma -\sigma\| +\|\Pi\sigma -\sigma_h\| + h\|\sigma\|_1 + h\|\g\|_1 \right) \nonumber  \\
& \qquad  \leq C\left( h\|\sigma\|_1 + h\|\g\|_1 + \epsilon\|\Qgh\g-\g_h\| \right),
\end{align*}
where we used \eqref{prop-Pi}, \eqref{q-err-h1-eq1}, \eqref{q-err-h1-eq3}, and  \eqref{stress-bound-msmfe1}.
Choosing $\epsilon$ small enough in the above, we obtain
\begin{align}
    \|\Quh u-u_h\|+\|\Qgh \g-\g_h\| \le Ch(\|\sigma\|_1 + h\|\g\|_1), \label{bound-displ-rotat}
\end{align}
which, combined with \eqref{stress-bound-msmfe1}, gives
\begin{align}
    \|\Pi\s-\s_h\| \le Ch(\|\s\|_1 + \|\g\|_1).\label{bound-l2-stress}
\end{align}
Also, using \eqref{eq2-err-msmfe1} and \eqref{prop-div} we get
\begin{align}
\|\dvrg(\s-\s_h)\| \le \|\dvr(\Pi\s-\s)\|\le Ch\|\dvr\s\|_1. \label{bound-div}
\end{align}
The assertion of the theorem for the MSMFE-1 method follows from combining
\eqref{bound-displ-rotat}--\eqref{bound-div} and using
\eqref{prop-Quh}--\eqref{prop-Qgh}. The proof the MSMFE-0 method follows
from the above argument by omitting the
quadrature error terms $\del(\cdot,\cdot)$ in
\eqref{eq1-err-msmfe1}--\eqref{eq3-err-msmfe1}. 
\end{proof}
{\color{black}
\begin{remark}
The error analysis for the modified MSMFE-1 method 
\eqref{scaled-msmfe-1-1}--\eqref{scaled-msmfe-1-3} based on the scaled
rotation $\tilde \gamma = A^{-1}\gamma$ follows along
the same lines. The resulting error estimate is 
\begin{align} \label{error-estimate-scaled}
\|\sigma-\sigma_h\|_{\dvrg} + \| u-u_h\|+\| \tilde\g-\tilde\g_h\|  
\leq Ch(\|\sigma\|_1 + \|\dvr \sigma\|_1 + \|u\|_1+\|\tilde\g\|_1).
\end{align}
This bound indicates that the modified method has an advantage if
$\tilde\gamma$ is smoother than $\gamma$, which is the case when
$A$ is discontinuous.
\end{remark}
}


\subsection{Second order convergence for the displacement}
We next prove superconvergence for the displacement. 
The following bounds on the quadrature error will be used in the analysis.
\begin{lemma} \label{q-err-h2-lem}
Let $A\in W^{2,\infty}_{\Tc_h}$. There exists a constant $C$ independent of $h$ such that
for all $\c,\t \in \X_h$, 
    \begin{align}
        | \tet\inp[A\c]{\t} | \le C\sum_{E\in\Tc_h} h^2\|\c\|_{1,E}\|\t\|_{1,E}, \label{q-err-h2-1}
    \end{align}
and for all $\xi \in \W^1_h$,
    \begin{align}
        | \del\inp[\t]{\xi} | \le C\sum_{E\in\Tc_h} h^2\|\t\|_{1,E}\|\xi\|_{1,E}. \label{q-err-h2-2}
    \end{align}
\end{lemma}
\begin{proof}
On any element $E$, using Lemma \ref{0-q-err-const} we have
    \begin{align*}
        \tet_E\inp[A\c]{\t} &= \tet_E\inp[(A-\bar{A})(\c - \bar{\c})]{\t} + \tet_E\inp[(A-\bar{A})\bar{\c}]{\t - \bar{\t}} + \tet_E\inp[A\bar{\c}]{\bar{\t}} + \tet_E\inp[\bar{A}(\c-\bar{\c})]{\t - \bar{\t}} \equiv \sum_{j=1}^4I_j.
    \end{align*}
Using \eqref{mean-value}, we obtain
$$
I_1 + I_2 + I_4 \le C h^2\|A\|_{1,\infty,E}\|\c\|_{1,E}\|\t\|_{1,E},
$$
while, using that the quadrature rule is exact for linears, the 
Bramble-Hilbert lemma \cite{ciarlet2002finite} gives
\begin{align}
|\tet_E\inp[A\bar{\c}]{\bar{\t}}| \le C h^2|A\bar\c|_{2,E}\|\bar\t\|_E
\le Ch^2|A|_{2,\infty, E}\|\c\|_E\|\t\|_E,
\end{align}
which implies \eqref{q-err-h2-1}. Similarly, using Lemma \ref{0-q-err-const} and
\eqref{mean-value}, we have
$$
\del_E\inp[\t]{\xi} = \del_E\inp[\t - \bar{\t}]{\xi - \bar{\xi}}
\le Ch^2\|\t\|_{1,E}\|\xi\|_{1,E},
$$
which implies \eqref{q-err-h2-2}.
\end{proof}

The superconvergence proof is based on a duality argument. We consider the 
auxiliary problem
    \begin{align}
    \begin{aligned}
        & \psi = A^{-1}\epsilon(\phi), \quad 
        \dvrg \psi = (\Quh u - u_h)  \quad \mbox {in } \O, \\
        & \phi = 0                    \quad \mbox{on } \Gd, \quad
        \psi\,n = 0                   \quad \mbox{on } \Gn, \label{aux}
    \end{aligned}
    \end{align}
and assume that it is $H^2$-elliptic regular:
	\begin{align}
	    \|\phi\|_2 \le \|\Quh u - u_h\| \label{el-reg}.
	\end{align}
Sufficient conditions for \eqref{el-reg} can be found in 
\cite{grisvard1985elliptic,lions1972non,ciarlet2002finite}.

\begin{theorem}\label{sc-p0}
Let $A\in W^{2,\infty}_{\Tc_h}$ and $A^{-1} \in W^{1,\infty}_{\Tc_h}$.
Assuming $H^2$-elliptic regularity \eqref{el-reg},
then for the MSMFE-0 and MSMFE-1 methods, there
exists a constant C independent of $h$ such that
\begin{align}
 \|\Quh u-u_h\| \le C h^2\left( \|\s\|_1 + \|\g\|_1 + \|\dvr\s\|_1 \right).
    \end{align}
\end{theorem}
\begin{proof}
We present the argument for the MSMFE-1 method. The proof for the
MSMFE-0 method follows by omitting the the quadrature error term
$\del(\cdot,\cdot)$. The error equation \eqref{err-m1-1} can be 
written as
\begin{align}
\inp[A(\s-\s_h)]{\t} + \inp[\Quh u-u_h]{\dvr \t} + \inp[\g - \g_h]{\t} 
+ \tet\inp[A\s_h]{\t} + \del\inp[\t]{\g_h} = 0.
\end{align}
Taking $\t = \Pi A^{-1}\epsilon(\phi)$ in the equation above, we get
\begin{align}
\begin{aligned} \label{bound-u-uh}
\|\Quh u-u_h\|^2 & = -\inp[A(\s-\s_h)]{\Pi A^{-1}\epsilon(\phi)} 
- \inp[\g - \g_h]{\Pi A^{-1}\epsilon(\phi)} \\
& \qquad - \tet\inp[A\s_h]{\Pi A^{-1}\epsilon(\phi)}
- \del\inp[\Pi A^{-1}\epsilon(\phi)]{\g_h}.
\end{aligned}
\end{align}
For the first term on the right, we have
\begin{align}
\begin{aligned}
\inp[A(\s-\s_h)]{\Pi A^{-1}\epsilon(\phi)} &= 
\inp[A(\s-\s_h)]{\Pi A^{-1}\epsilon(\phi) - A^{-1}\epsilon(\phi)} 
+ \inp[\s-\s_h]{\nabla\phi - \skew(\nabla \phi)} \\
&= \inp[A(\s-\s_h)]{\Pi A^{-1}\epsilon(\phi) - A^{-1}\epsilon(\phi)} 
- \inp[\dvr(\s-\s_h)]{\phi - \Quh\phi} \\
& \qquad
- \inp[\s -\s_h]{\skew(\nabla\phi) - \Qgh\skew(\nabla\phi)}
+ \delta(\s-\s_h,\Qgh\skew(\nabla\phi))
\\
&\le C h^2\left( \|\s\|_1 + \|\g\|_1 + \|\dvr \s\|_1 \right)\|\phi\|_2, \label{sc-simp-rhs-1}
	\end{aligned}
	\end{align}
where we used \eqref{prop-Quh}--\eqref{prop-Pi}, 
\eqref{continuity-l2-proj}, \eqref{q-err-h1-eq2}, \eqref{error-estimate},
\eqref{err-m1-2}, and \eqref{err-m1-3}. For the second term on the right in
\eqref{bound-u-uh} we have
	\begin{align}
	\begin{aligned}
	    \inp[\g - \g_h]{\Pi A^{-1}\epsilon(\phi)} 
&= \inp[\g - \g_h]{\Pi A^{-1}\epsilon(\phi) - A^{-1}\epsilon(\phi) } 
+ \inp[\g - \g_h]{A^{-1}\epsilon(\phi) } \\
&= \inp[\g - \g_h]{\Pi A^{-1}\epsilon(\phi) - A^{-1}\epsilon(\phi) } \\
&\le Ch^2\left( \|\s\|_1 + \|\g\|_1 \right)\|\phi\|_2, \label{sc-simp-rhs-3}
	\end{aligned}
	\end{align}
where the second equality is due to the skew-symmetry of $(\g - \g_h)$ and 
the symmetry of $A^{-1}\epsilon(\phi)$, and the inequality follows from 
\eqref{prop-Pi} and \eqref{error-estimate}. For the third term
on the right in \eqref{bound-u-uh} we write, using \eqref{q-err-h2-1}
\begin{align}
\begin{aligned}
\tet\inp[A\s_h]{\Pi A^{-1}\epsilon(\phi)} &\le C \sum_{E\in \Tc_h} h^2 \|\s_h\|_{1,E} \|\Pi A^{-1}\epsilon(\phi)\|_{1,E} \\
	        &\le C\sum_{E\in \Tc_h} h^2 \left( \|\s_h - \Pi\s\|_{1,E} + \|\Pi\s\|_{1,E} \right)\| A^{-1}\epsilon(\phi)\|_{1,E} \\
	        &\le C\sum_{E\in \Tc_h} h^2 \left( h^{-1}\|\s_h-\Pi\s\|_E + \|\s\|_{1,E} \right) \|\epsilon(\phi)\|_{1,E} \\
	        &\le C h^2 \left(\|\s\|_{1} + \|\g\|_{1} \right) \|\phi\|_{2},\label{sc-simp-rhs-4}
	\end{aligned}
	\end{align}
where we used \eqref{mixed-proj-continuity}, \eqref{inverse}, and 
\eqref{error-estimate}. Similarly, for the last term on the right in 
\eqref{bound-u-uh}, using \eqref{q-err-h2-2}, 
\eqref{continuity-l2-proj}, \eqref{inverse}, and \eqref{error-estimate}, 
we have
	\begin{align}
	\begin{aligned}
	    \del\inp[\Pi A^{-1}\epsilon(\phi)]{\g_h} 
	    &\le C\sum_{E\in\Tc_h} h^2 \left( \|\g_h - \Qgh \g\|_{1,E} + \|\Qgh\g\|_{1,E} \right) \|A^{-1}\epsilon(\phi)\|_{1,E} \\
	    &\le Ch^2 \left( \|\s\|_1 + \|\g\|_1 \right) \|\phi\|_2.
\label{sc-simp-rhs-5}
	\end{aligned}
	\end{align}
The statement of the theorem follows by combining 
\eqref{bound-u-uh}--\eqref{sc-simp-rhs-5} and elliptic regularity 
\eqref{el-reg}.
\end{proof}

\section{Numerical results}
We present several numerical experiments confirming the theoretical
convergence rates. We used FEniCS Project \cite{LoggMardalEtAl2012a}
for the implementation of the MSMFE-0 and MSMFE-1 methods on
simplicial grids in 2D and 3D.  Both methods have been implemented
using the rotation variable $p_h = \Xi^{-1}(\g_h)$, where $\Xi$ is
defined in \eqref{skew-extra}. For example, using
\eqref{asym-identity}, the MSMFE-1 method
\eqref{msmfe-1-1}--\eqref{msmfe-1-3} can be written as
    \begin{align}
	(A\sigma_h,\tau)_Q + (u_h,\dvr{\tau}) + (p_h, \asym\tau)_Q &= \gnp[g]{\tau}_{\Gd}, &\tau &\in \X_h, \label{msmfe-actual-1}\\
	(\dvr \sigma_h,v) &= (f,v), &v &\in V_h, \label{msmfe-actual-2} \\
	(\asym\sigma_h,w)_Q &= 0, &w &\in \Xi^{-1}(\W_h^1), \label{msmfe-actual-3}
\end{align}
with a similar formulation for the MSMFE-0 method. Note that the rotation a scalar in
$\Pc_k$ in 2D, and a vector in $(\Pc_k)^3$ in 3D, with $k=0,1$ for MSMFE-0 and MSMFE-1, 
respectively.

In the first example we study the convergence of the proposed methods in 2D. We consider
a test case from \cite{arnold2015mixed} on 
the unit square with homogeneous Dirichlet boundary conditions
and analytical solution given by
$$
u = \begin{pmatrix} \cos(\pi x)\sin(2\pi y) \\ \cos(\pi y)\sin(\pi x) \end{pmatrix}.
$$ 
The body force is then determined using Lam\'{e} coefficients $\lambda
= 123,\, \mu = 79.3$. The computed solution is shown in
Figure~\ref{fig:1_1}. Since we use $p_h = \Xi^{-1}(\g_h)$ for the
Lagrange multiplier, the errors are also computed using this
variable. However, it is clear that the operator $\Xi$ does not
introduce extra numerical error.

In Table \ref{tab:1} we show errors and convergence rates on a
sequence of mesh refinements, computed using the MSMFE-0 and MSMFE-1
methods, including displacement superconvergence. All rates are in
accordance with the error analysis presented in the previous section.
We note that the MSMFE-1 method with linear rotations exhibits
convergence for the rotation of order $O(h^{1.5})$, slightly higher
than the theoretical result.
\begin{figure}[ht!]
	\centering
	\begin{subfigure}[b]{0.24\textwidth}
		\includegraphics[width=\textwidth]{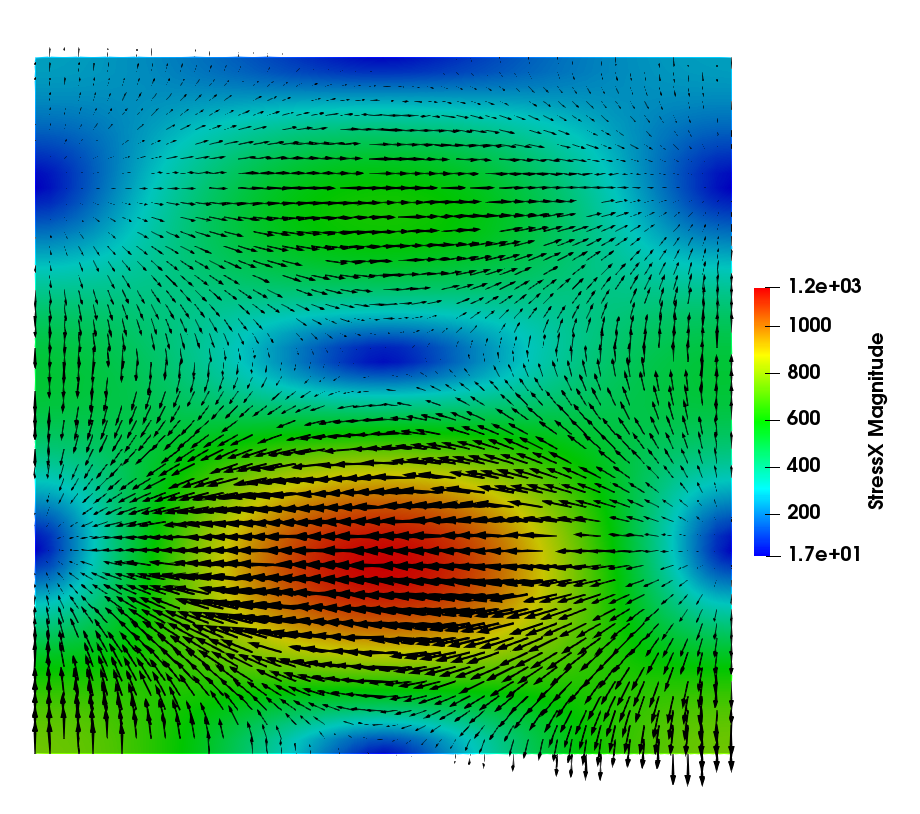}
		\caption{$x$-component of stress}
		\label{fig:1_1}
	\end{subfigure}
	\begin{subfigure}[b]{0.24\textwidth}
		\includegraphics[width=\textwidth]{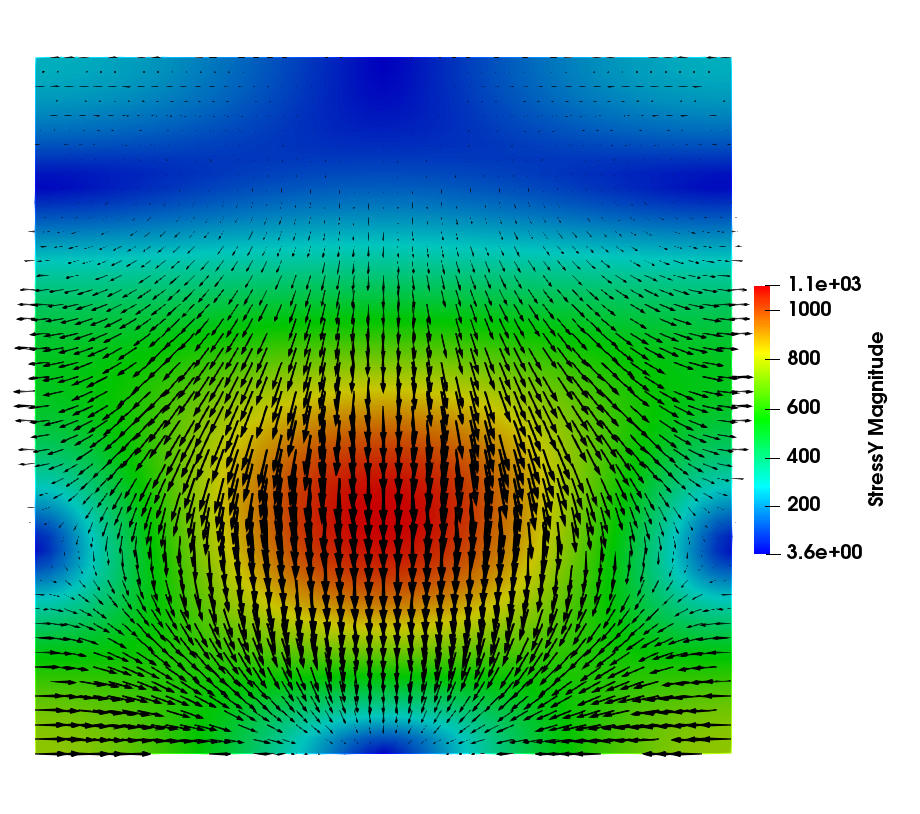}
		\caption{$y$-component of stress}
		\label{fig:1_2}
	\end{subfigure}
	\begin{subfigure}[b]{0.24\textwidth}
		\includegraphics[width=\textwidth]{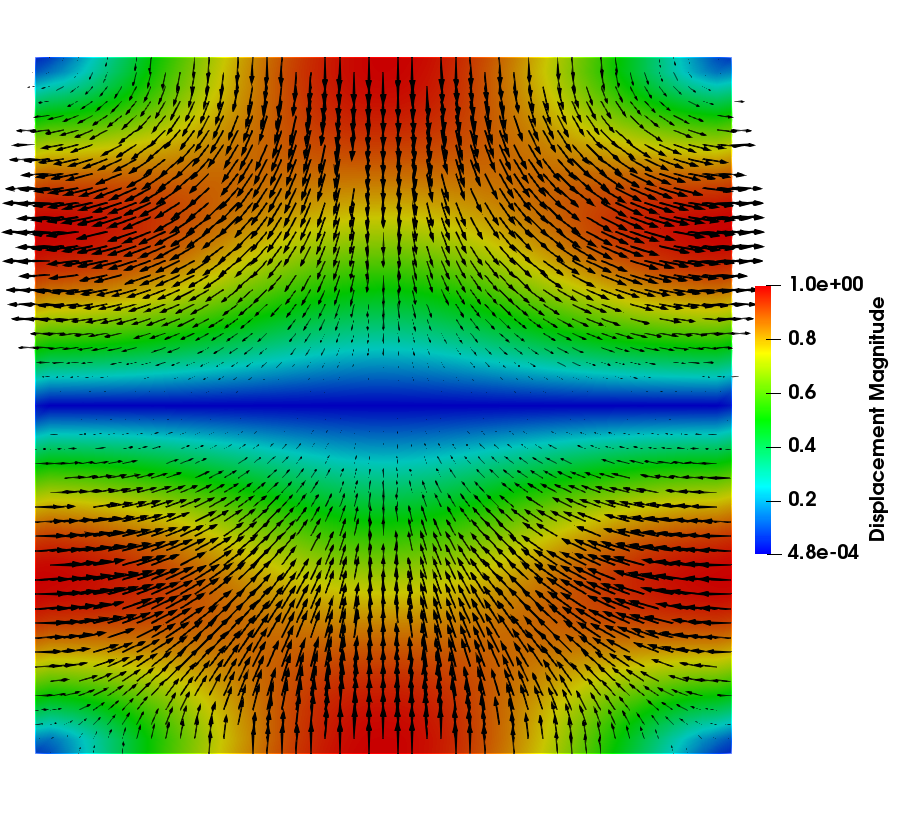}
		\caption{Displacement}
		\label{fig:1_3}
	\end{subfigure}
	\begin{subfigure}[b]{0.24\textwidth}
		\includegraphics[width=\textwidth]{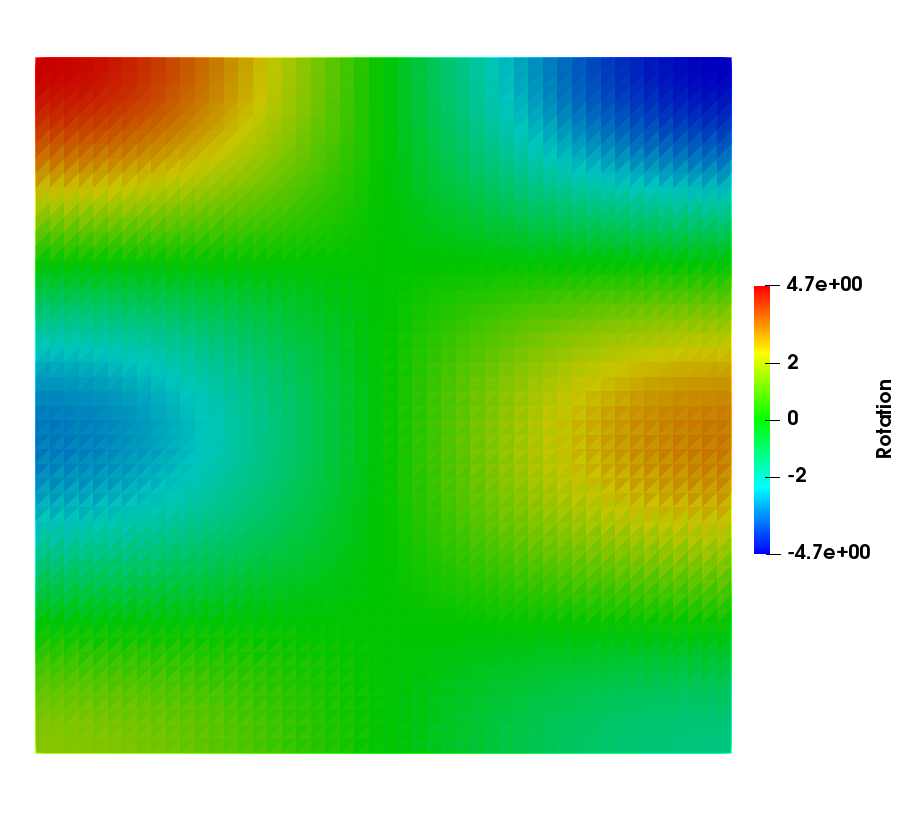}
		\caption{Rotation}
		\label{fig:1_4}
	\end{subfigure}
	\caption{Computed solution for Example 1, $h=1/32$}\label{fig:4}
\end{figure}
\begin{table}[ht!]
\begin{center}
\begin{tabular}{c|cc|cc|cc|cc|cc}
	\hline
	\multicolumn{11}{c}{MSMFE-0} \\
	\hline
	      & \multicolumn{2}{c|}{$\|\sigma - \sigma_h\|$} & \multicolumn{2}{c|}{$\|\dvr(\sigma - \sigma_h)\|$} &      \multicolumn{2}{c|}{$\|u - u_h\|$}   & \multicolumn{2}{c|}{$\|\Quh u - u_h\|$}   &         \multicolumn{2}{c}{$\|p -p_h\|$}         \\
	      $h$   &  error   &                 rate                  &  error   &                    rate      &  error   &                    rate                    &  error   &                 rate                  &  error   &                    rate                    \\\hline
1/2	&	8.01E-01	&	--	&	8.98E-01	&	--&	8.37E-01	&	--&	8.24E-01	&	--&	1.02E+00	&	--	\\
1/4	&	3.58E-01	&	1.17	&	4.26E-01	&	1.09	&	3.50E-01	&	1.27	&	1.82E-01	&	2.34	&	5.03E-01	&	1.02	\\
1/8	&	1.53E-01	&	1.23	&	1.99E-01	&	1.10	&	1.73E-01	&	1.02	&	4.70E-02	&	1.96	&	3.13E-01	&	0.69	\\
1/16	&	7.03E-02	&	1.12	&	9.84E-02	&	1.02	&	8.67E-02	&	1.00	&	1.20E-02	&	1.97	&	1.71E-01	&	0.87	\\
1/32	&	3.42E-02	&	1.04	&	5.00E-02	&	0.98	&	4.35E-02	&	0.99	&	3.03E-03	&	1.99	&	8.78E-02	&	0.96	\\
1/64	&	1.70E-02	&	1.01	&	2.60E-02	&	0.95	&	2.18E-02	&	1.00	&	7.59E-04	&	2.00	&	4.42E-02	&	0.99	\\\hline
	\hline
	\multicolumn{11}{c}{MSMFE-1} \\
	\hline

	 & \multicolumn{2}{c|}{$\|\sigma - \sigma_h\|$} & \multicolumn{2}{c|}{$\|\dvr(\sigma - \sigma_h)\|$} &      \multicolumn{2}{c|}{$\|u - u_h\|$}   & \multicolumn{2}{c|}{$\|\Quh u - u_h\|$}   &         \multicolumn{2}{c}{$\|p -p_h\|$}         \\
	 $h$   &  error   &                 rate                  &  error   &                    rate      &  error   &                    rate                    &  error   &                 rate                  &  error   &                    rate                    \\\hline
1/2	&	7.96E-01	&	--	&	9.01E-01	&	--&	8.60E-01	&	--	&	8.47E-01	&	--&	9.95E-01	&	--	\\
1/4	&	3.67E-01	&	1.13	&	4.26E-01	&	1.09	&	3.55E-01	&	1.29	&	1.95E-01	&	2.28	&	4.55E-01	&	1.12	\\
1/8	&	1.56E-01	&	1.23	&	1.93E-01	&	1.14	&	1.76E-01	&	1.01	&	5.67E-02	&	1.78	&	1.68E-01	&	1.44	\\
1/16	&	7.11E-02	&	1.14	&	9.34E-02	&	1.05	&	8.75E-02	&	1.01	&	1.55E-02	&	1.87	&	5.37E-02	&	1.65	\\
1/32	&	3.43E-02	&	1.05	&	4.66E-02	&	1.00	&	4.37E-02	&	1.00	&	4.01E-03	&	1.95	&	1.66E-02	&	1.70	\\
1/64	&	1.70E-02	&	1.02	&	2.37E-02	&	0.98	&	2.18E-02	&	1.00	&	1.02E-03	&	1.98	&	5.26E-03	&	1.66	\\\hline
\end{tabular}
\end{center}
\caption{Relative errors and convergence rates for Example 1, triangles.} \label{tab:1} 
\end{table}

The second test case illustrates the performance of the methods in 3D. We consider the 
unit cube with homogeneous Dirichlet boundary conditions,
analytical solution given by
\begingroup
\renewcommand*{\arraystretch}{1.5}
\begin{align}
    u = \begin{pmatrix}
        0 \\
        -(e^x - 1)(y - \cos(\frac{\pi}{12})(y-\frac{1}{2}) + \sin(\frac{\pi}{12})(z-\frac{1}{2})-\frac{1}{2}) \\
        -(e^x - 1)(z - \sin(\frac{\pi}{12})(y-\frac{1}{2}) - \cos(\frac{\pi}{12})(z-\frac{1}{2})-\frac{1}{2})
 \end{pmatrix},
\end{align}
\endgroup
and Lam\'{e} coefficients $\lambda = \mu = 100$. The computed solution is shown in
Figure~\ref{fig:5}. In Table \ref{tab:2}
we show errors and convergence rates for both methods on a sequence of
mesh refinements. Again we observe that the numerical results verify
the theoretical convergence rates.
\begin{figure}[ht!]
	\centering
	\begin{subfigure}[b]{0.194\textwidth}
		\includegraphics[width=\textwidth]{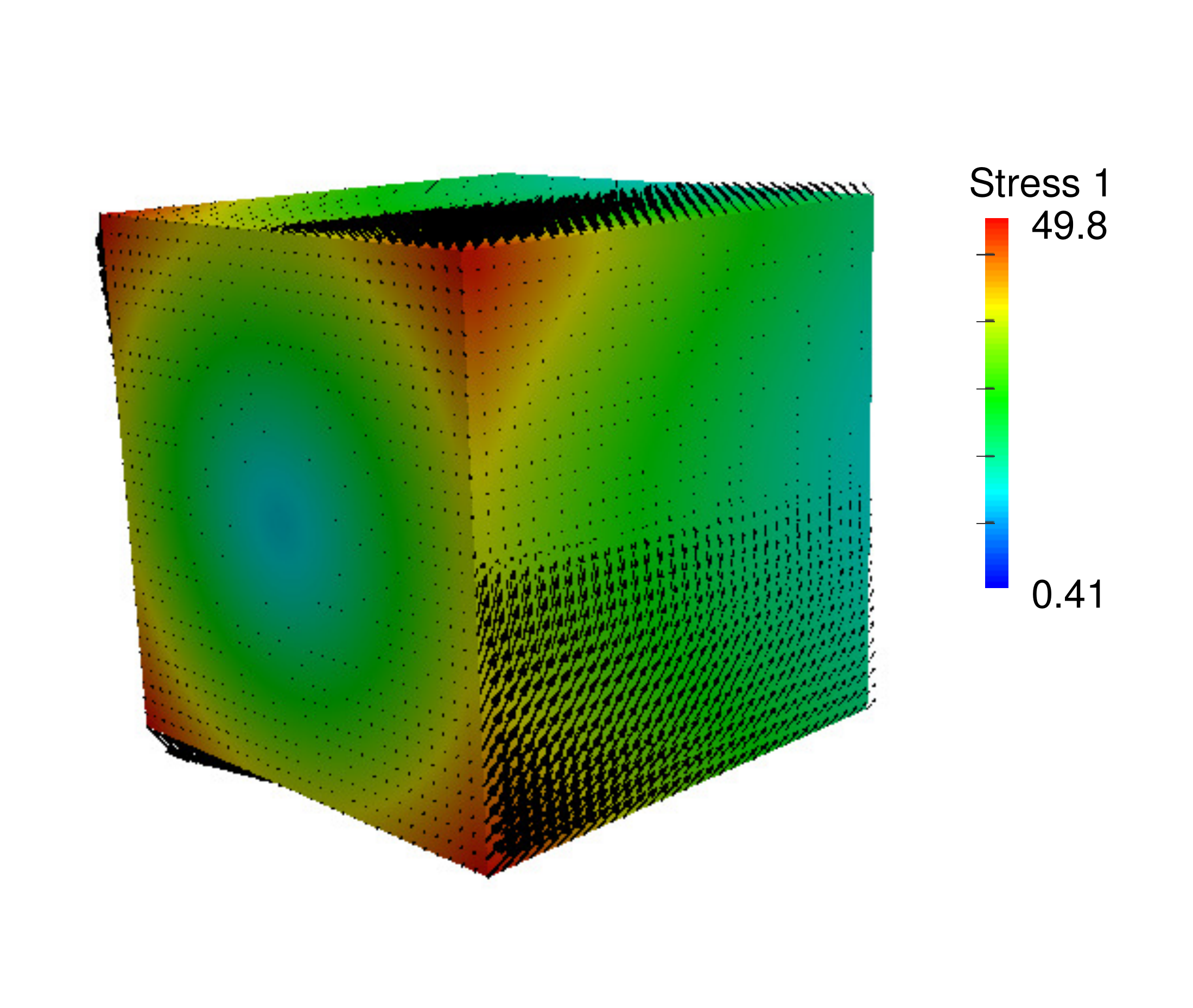}
		\caption{$x$-stress}
		\label{fig:2_1}
	\end{subfigure}
	\begin{subfigure}[b]{0.194\textwidth}
		\includegraphics[width=\textwidth]{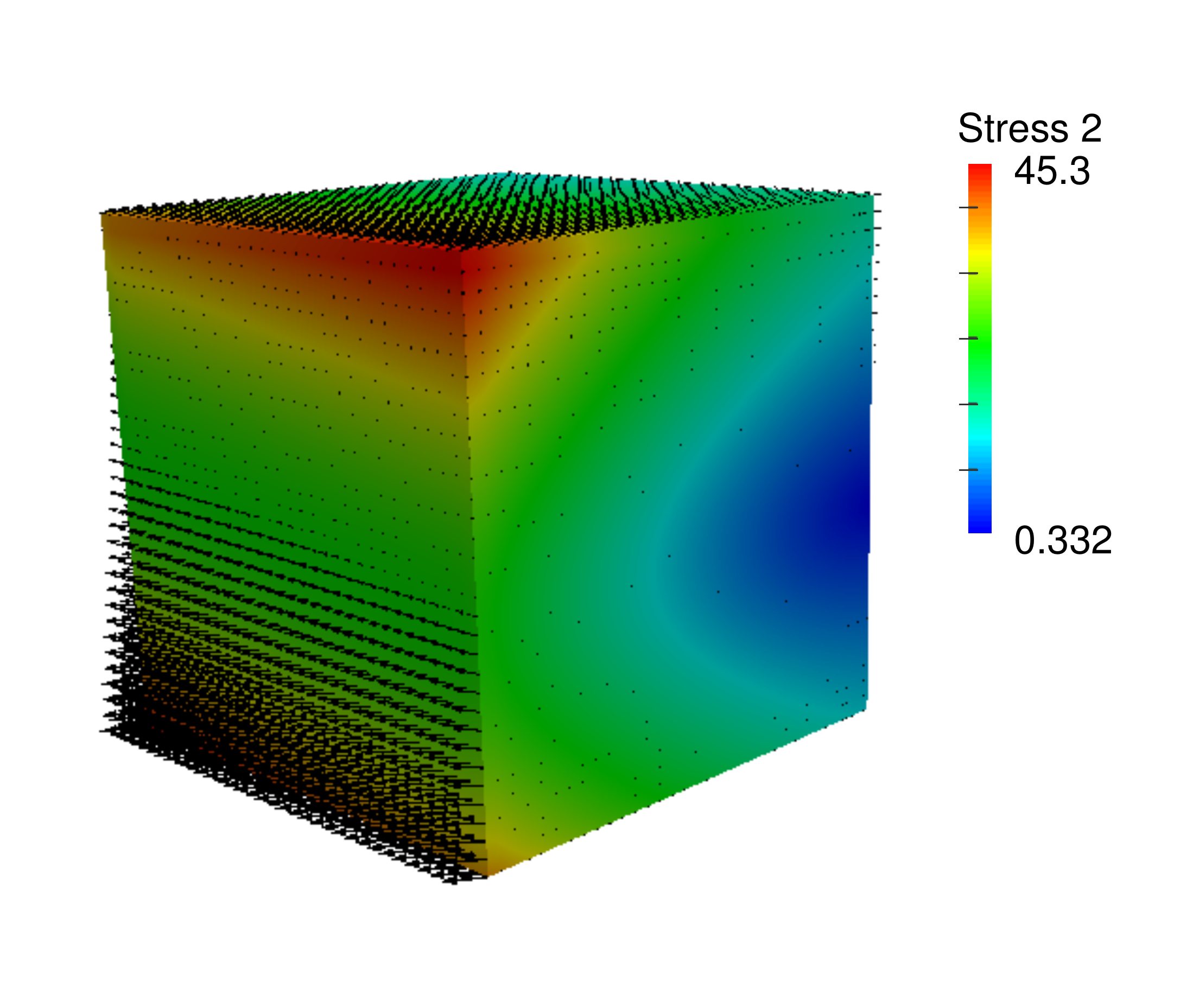}
		\caption{$y$-stress}
		\label{fig:2_2}
	\end{subfigure}
		\begin{subfigure}[b]{0.194\textwidth}
		\includegraphics[width=\textwidth]{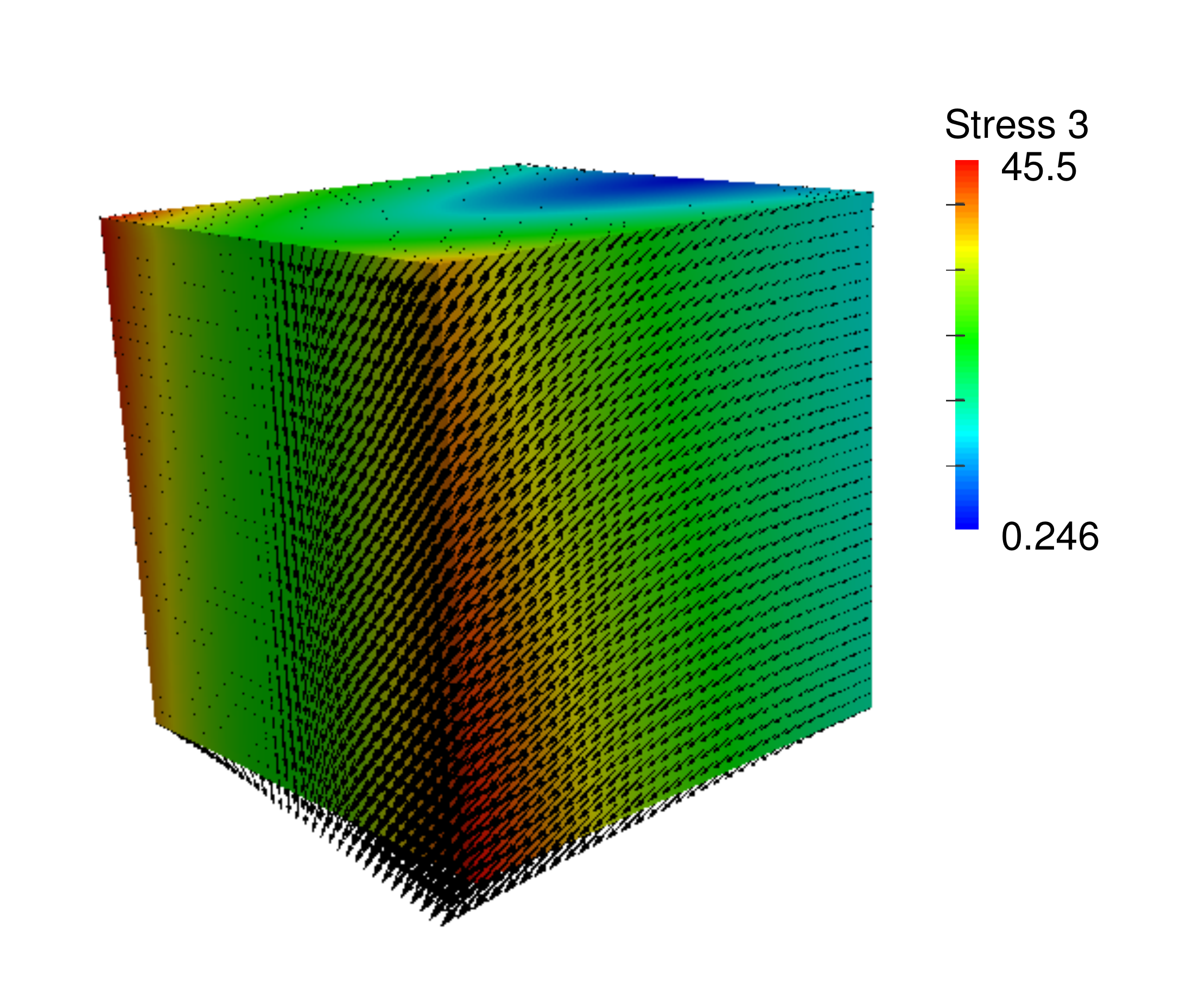}
		\caption{$z$-stress}
		\label{fig:2_3}
	\end{subfigure}
	\begin{subfigure}[b]{0.194\textwidth}
		\includegraphics[width=\textwidth]{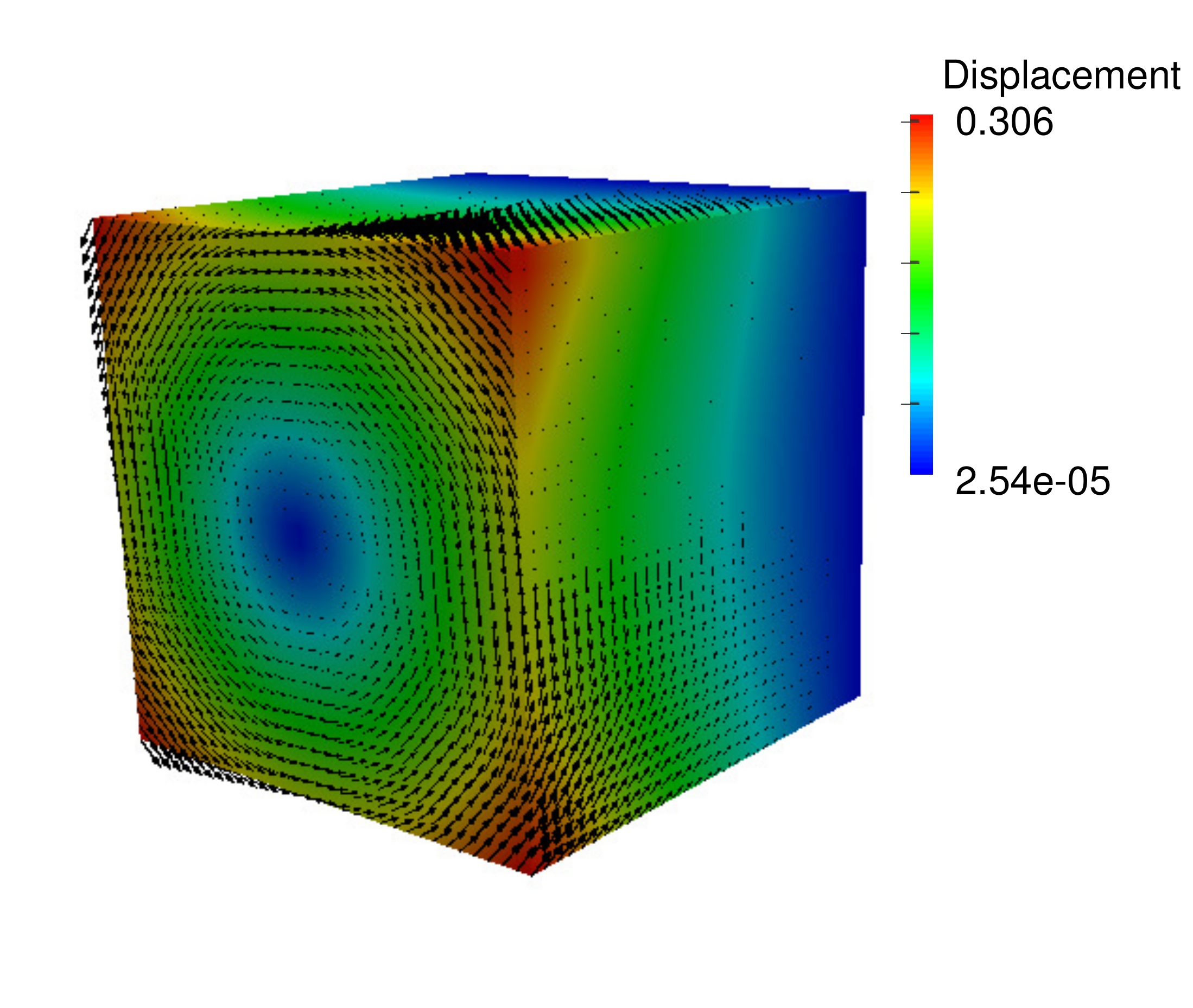}
		\caption{Displacement}
		\label{fig:2_4}
	\end{subfigure}
	\begin{subfigure}[b]{0.194\textwidth}
		\includegraphics[width=\textwidth]{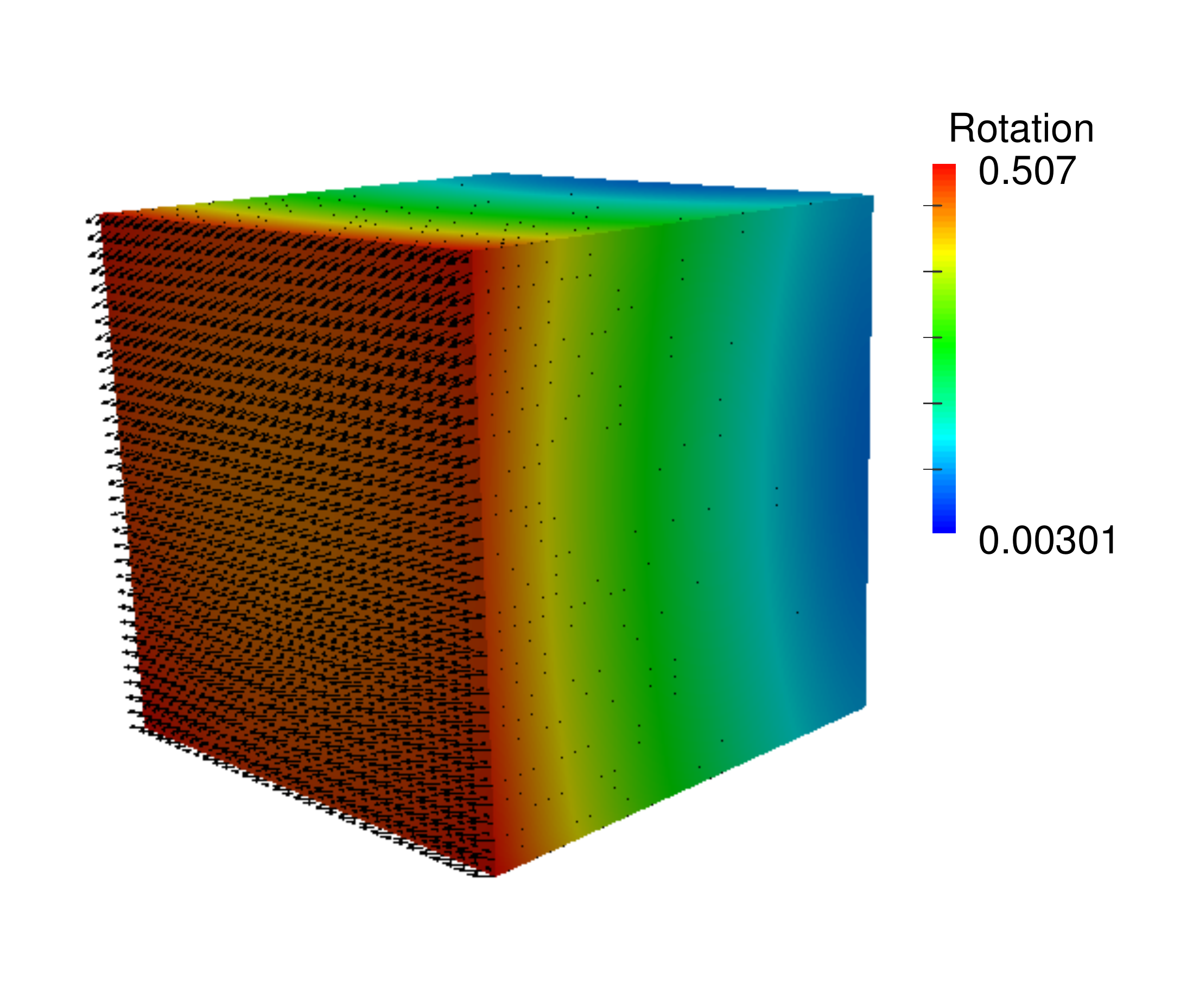}
		\caption{Rotation}
		\label{fig:2_5}
	\end{subfigure}
	\caption{Computed solution for Example 2, $h=1/32$}\label{fig:5}
\end{figure}
\begin{table}[ht!]
\begin{center}
\begin{tabular}{c|cc|cc|cc|cc|cc}
	\hline
	\multicolumn{11}{c}{MSMFE-0} \\
	\hline
	      & \multicolumn{2}{c|}{$\|\sigma - \sigma_h\|$} & \multicolumn{2}{c|}{$\|\dvr(\sigma - \sigma_h)\|$} &      \multicolumn{2}{c|}{$\|u - u_h\|$}   & \multicolumn{2}{c|}{$\|\Quh u - u_h\|$}   &         \multicolumn{2}{c}{$\|p -p_h\|$}         \\
	  $h$   &  error   &                 rate                  &  error   &                    rate      &  error   &                    rate                    &  error   &                 rate                  &  error   &                    rate                    \\\hline
1/2	&	4.46E-01	&	--&	2.45E-01	&	--	&	4.15E-01	&	--&	1.32E-01	&	--&	2.41E-01	&	--	\\
1/4	&	1.96E-01	&	1.19	&	1.21E-01	&	1.02	&	2.06E-01	&	1.01	&	3.11E-02	&	1.98	&	1.20E-01	&	1.00	\\
1/8	&	9.08E-02	&	1.11	&	6.02E-02	&	1.01	&	1.03E-01	&	1.00	&	7.72E-03	&	1.98	&	6.01E-02	&	1.00	\\
1/16	&	4.40E-02	&	1.05	&	3.01E-02	&	1.00	&	5.14E-02	&	1.00	&	1.94E-03	&	1.99	&	2.99E-02	&	1.00	\\
1/32	&	2.17E-02	&	1.02	&	1.51E-02	&	1.00	&	2.57E-02	&	1.00	&	4.85E-04	&	2.00	&	1.49E-02	&	1.00	\\\hline \hline
	\multicolumn{11}{c}{MSMFE-1} \\
	\hline
	      & \multicolumn{2}{c|}{$\|\sigma - \sigma_h\|$} & \multicolumn{2}{c|}{$\|\dvr(\sigma - \sigma_h)\|$} &      \multicolumn{2}{c|}{$\|u - u_h\|$}   & \multicolumn{2}{c|}{$\|\Quh u - u_h\|$}   &         \multicolumn{2}{c}{$\|p -p_h\|$}         \\
	      $h$   &  error   &                 rate                  &  error   &                    rate      &  error   &                    rate                    &  error   &                 rate                  &  error   &                    rate                    \\\hline
1/2	&	5.40E-01	&	--	&	2.45E-01	&	--	&	4.20E-01	&	--	&	1.55E-01	&	--	&	2.38E-01	&	--	\\
1/4	&	2.42E-01	&	1.16	&	1.21E-01	&	1.02	&	2.07E-01	&	1.02	&	4.04E-02	&	1.83	&	1.00E-01	&	1.24	\\
1/8	&	1.09E-01	&	1.15	&	6.02E-02	&	1.01	&	1.03E-01	&	1.01	&	1.07E-02	&	1.89	&	3.93E-02	&	1.35	\\
1/16	&	5.05E-02	&	1.12	&	3.01E-02	&	1.00	&	5.14E-02	&	1.00	&	2.81E-03	&	1.93	&	1.47E-02	&	1.42	\\
1/32	&	2.39E-02	&	1.08	&	1.51E-02	&	1.00	&	2.57E-02	&	1.00	&	7.20E-04	&	1.96	&	5.38E-03	&	1.45	\\\hline
\end{tabular}
\end{center}
\caption{Relative errors and convergence rates for Example 2, tetrahedra.} \label{tab:2} 
\end{table}

Our third example, taken from \cite{Jan-IJNME}, demonstrates the performance
of the MSMFE methods for discontinuous materials. We consider a
$3\times 3$ partitioning of the unit square and introduce heterogeneity in the 
center block through
\begin{align*}
	\chi(x,y) = 
	\begin{cases}
		1 \mbox{ if } \min(x,y) > \frac13 \mbox{ and } \max(x,y) < \frac23, \\
		0 \mbox{ otherwise.}
	\end{cases}
\end{align*}
We set $\kappa = 10^6$ to characterize the jump in
the Lam\'{e} coefficients and take $ \lambda = \mu = (1-\chi) + \kappa\chi$.
We choose a continuous displacement solution as
\begin{align*}
	u = \frac{1}{(1-\chi) + \kappa\chi}
	\begin{pmatrix}
		\sin(3\pi x)\sin(3\pi y) \\ \sin(3\pi x)\sin(3\pi y)
	\end{pmatrix},
\end{align*}
so that the stress is also continuous and independent of $\kappa$. The
body forces are recovered from the above solution using the governing
equations. We note that the rotation $\gamma = \skew(\nabla u)$ is
discontinuous. The MSMFE-0 method, which has discontinuous
displacements and rotations, handles properly the discontinuity in
these variables and exhibits first order convergence in all variables,
as as well as displacement superconvergence, see the top part of
Table~\ref{tab:3}. The MSMFE-1 method uses continuous rotations and
does not resolve the rotation discontinuity, which results in a
reduced convergence rate for the rotation, as well as the
stress. {\color{black} Instead, we can use the modified MSMFE-1
method \eqref{scaled-msmfe-1-1}--\eqref{scaled-msmfe-1-3} based
on the scaled rotation $\tilde \gamma = A^{-1} \gamma$, which 
in this case is continuous.} In terms of the implemented
method \eqref{msmfe-actual-1}--\eqref{msmfe-actual-3} with the reduced
rotation $p_h = \Xi^{-1}(\g_h)$, noting that $\tilde p =
\Xi^{-1}(\tilde \gamma)$, the third term in \eqref{msmfe-actual-1}
becomes $(\tilde p_h, \asym(A\tau))_Q$ and the term in
\eqref{msmfe-actual-3} becomes $(\asym(A\sigma_h),w)_Q$. The computed
solution with the modified MSMFE-1 method, including the scaled
rotation $\tilde p_h$, is shown in Figure~\ref{fig:6}. The bottom part
of Table~\ref{tab:3} indicates that the method exhibits the same order
of convergence for all variables as for smooth problems. 
\begin{figure}[ht!]
	\centering
	\begin{subfigure}[b]{0.194\textwidth}
		\includegraphics[width=\textwidth]{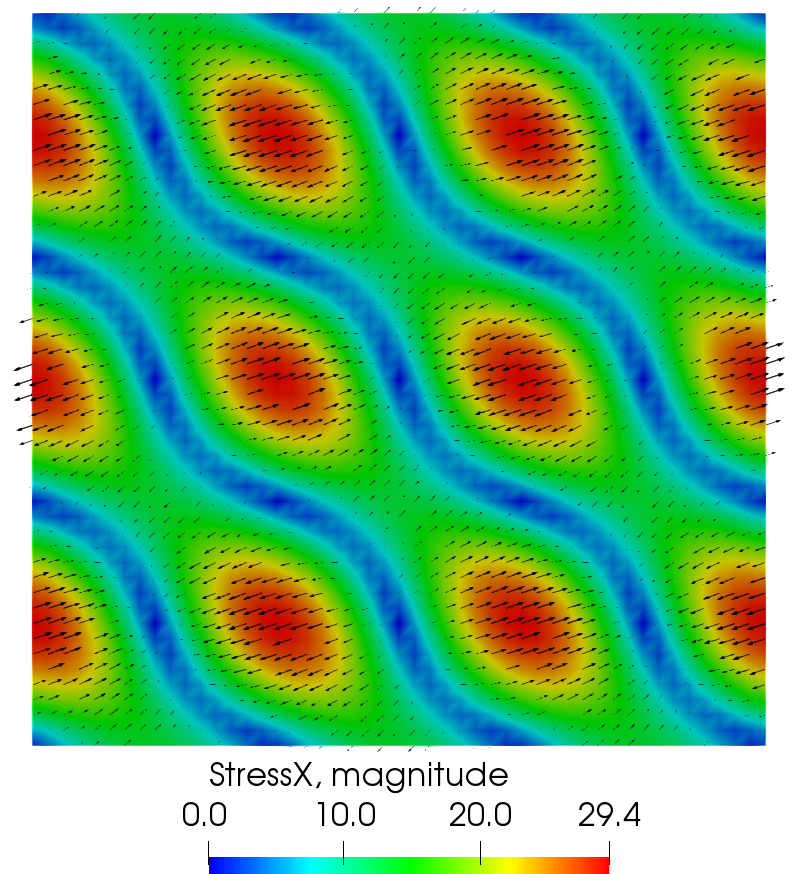}
		\caption{$x$-stress}
		\label{fig:6_1}
	\end{subfigure}
	\begin{subfigure}[b]{0.194\textwidth}
		\includegraphics[width=\textwidth]{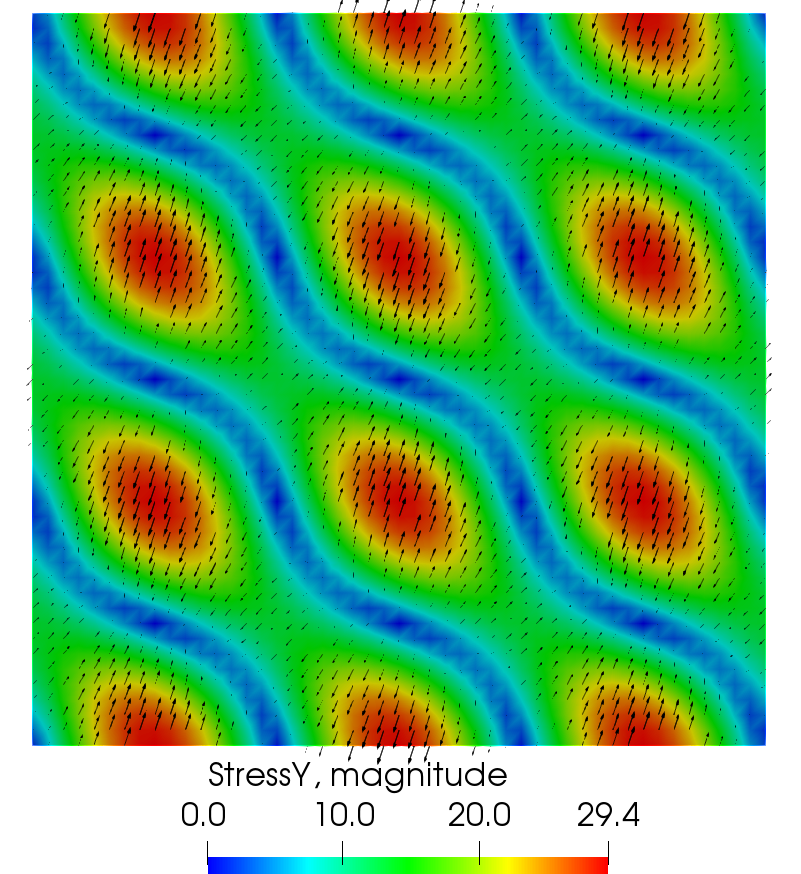}
		\caption{$y$-stress}
		\label{fig:6_2}
	\end{subfigure}
	\begin{subfigure}[b]{0.194\textwidth}
		\includegraphics[width=\textwidth]{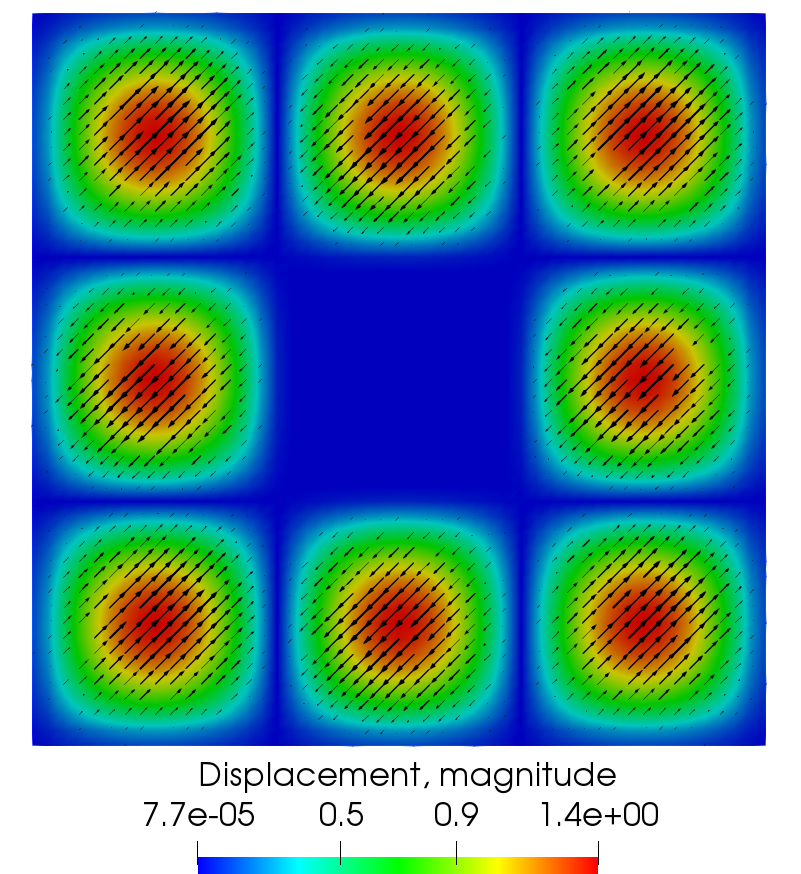}
		\caption{Displacement}
		\label{fig:6_3}
	\end{subfigure}
	\begin{subfigure}[b]{0.194\textwidth}
		\includegraphics[width=\textwidth]{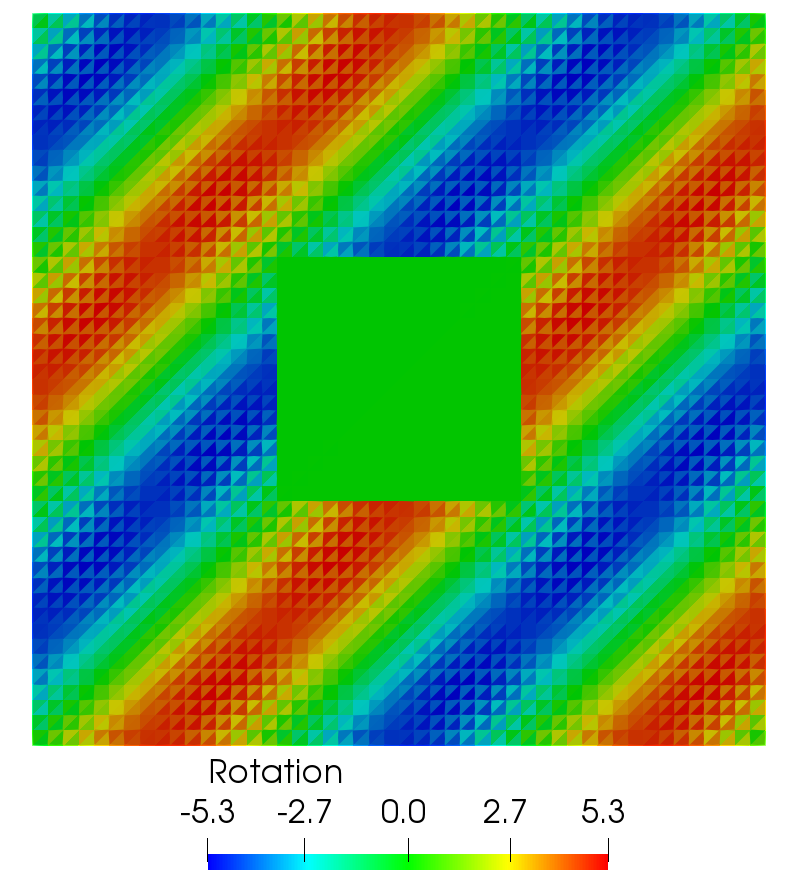}
		\caption{Rotation}
		\label{fig:6_4}
	\end{subfigure}
	\begin{subfigure}[b]{0.194\textwidth}
		\includegraphics[width=\textwidth]{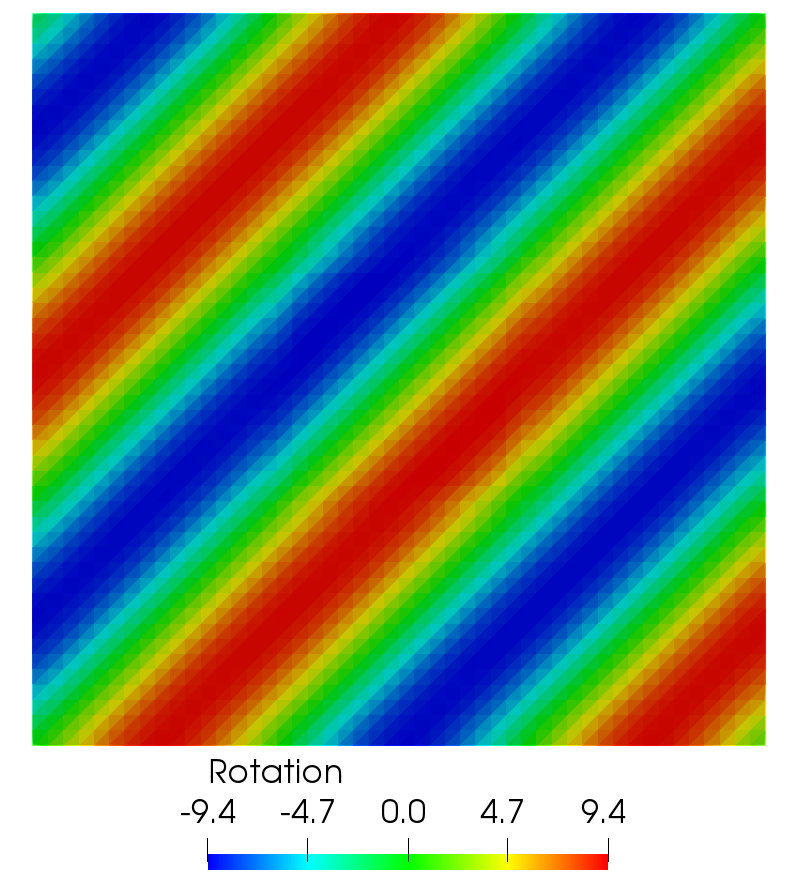}
		\caption{Scaled rotation}
		\label{fig:6_5}
	\end{subfigure}
	\caption{Computed solution for Example 3, $h=1/48$}\label{fig:6}
\end{figure}
\begin{table}[ht!]
	\begin{center}
		\begin{tabular}{c|cc|cc|cc|cc|cc}
			\hline
			\multicolumn{11}{c}{MSMFE-0} \\
			\hline
			& \multicolumn{2}{c|}{$\|\sigma - \sigma_h\|$} & \multicolumn{2}{c|}{$\|\dvr(\sigma - \sigma_h)\|$} &      \multicolumn{2}{c|}{$\|u - u_h\|$}   & \multicolumn{2}{c|}{$\|\Quh u - u_h\|$}   &         \multicolumn{2}{c}{$\|p -p_h\|$}         \\
			$h$   &  error   &                 rate                  &  error   &                    rate      &  error   &                    rate                    &  error   &                 rate                  &  error   &                    rate                    \\\hline
  1/3 & 1.27E+00 &   - & 1.20E+00 &   - & 1.61E+00 &   - & 1.49E+00 &   - & 1.46E+00 &   - \\
  1/6 & 6.97E-01 &   0.87 & 7.28E-01 &   0.73 & 5.87E-01 &   1.45 & 4.55E-01 &   1.71 & 6.50E-01 &   1.17 \\
  1/12 & 2.68E-01 &   1.38 & 3.33E-01 &   1.13 & 2.73E-01 &   1.10 & 1.19E-01 &   1.93 & 4.70E-01 &   0.47 \\
  1/24 & 1.05E-01 &   1.35 & 1.58E-01 &   1.07 & 1.33E-01 &   1.04 & 3.08E-02 &   1.95 & 2.76E-01 &   0.77 \\
  1/48 & 4.72E-02 &   1.16 & 7.79E-02 &   1.02 & 6.57E-02 &   1.01 & 7.79E-03 &   1.98 & 1.45E-01 &   0.93 \\
  1/96 & 2.28E-02 &   1.05 & 3.88E-02 &   1.01 & 3.28E-02 &   1.00 & 1.96E-03 &   1.99 & 7.34E-02 &   0.98 \\ \hline \hline
			\multicolumn{11}{c}{MSMFE-1 {\color{black} with scaled rotation}} \\
			\hline
			& \multicolumn{2}{c|}{$\|\sigma - \sigma_h\|$} & \multicolumn{2}{c|}{$\|\dvr(\sigma - \sigma_h)\|$} &      \multicolumn{2}{c|}{$\|u - u_h\|$}   & \multicolumn{2}{c|}{$\|\Quh u - u_h\|$}   &         \multicolumn{2}{c}{$\|\tilde{p} -\tilde{p}_h\|$}         \\
			$h$   &  error   &                 rate                  &  error   &                    rate      &  error   &                    rate                    &  error   &                 rate                  &  error   &                    rate                    \\\hline
			1/3	&	1.26E+00	&	-	&	1.20E+00	&	-	&	1.73E+00	&	-	&	1.59E+00	&	-	&	1.20E+00	&	-	\\
			1/6	&	6.82E-01	&	0.88	&	7.28E-01	&	0.73	&	5.74E-01	&	1.59	&	4.28E-01	&	1.89	&	5.46E-01	&	1.14	\\
			1/12	&	2.60E-01	&	1.39	&	3.33E-01	&	1.13	&	2.72E-01	&	1.08	&	1.17E-01	&	1.87	&	2.10E-01	&	1.38	\\
			1/24	&	1.03E-01	&	1.34	&	1.58E-01	&	1.07	&	1.33E-01	&	1.04	&	3.08E-02	&	1.92	&	6.68E-02	&	1.66	\\
			1/48	&	4.65E-02	&	1.14	&	7.79E-02	&	1.02	&	6.57E-02	&	1.01	&	7.90E-03	&	1.96	&	2.11E-02	&	1.66	\\
			1/96	&	2.26E-02	&	1.04	&	3.88E-02	&	1.01	&	3.28E-02	&	1.00	&	2.01E-03	&	1.98	&	6.95E-03	&	1.60	\\ \hline
		\end{tabular}
	\end{center}
	\caption{Relative errors and convergence rates for Example 3, triangles.} \label{tab:3} 
\end{table}

{\color{black}Our final example, similar to the one in
  \cite{Hansbo-Larson}, is to study the locking-free
  property of the MSMFE method.  We consider the MSMFE-1 method on the
  unit square domain with the following boundary conditions: $u = 0$
  at $y=0$, $\sigma\,n = 0$ at $x=0 \mbox{ and } x=1$, and 
  $(\sigma\,n) \cdot n = 0$, $(\sigma\,n) \cdot t = 1$ at $y=1$, 
  where $t$ denotes the unit
  tangential vector to the side. We recall that the Lam\'{e}
  coefficients are determined from the Young's modulus $E$ and the
  Poisson's ratio $\nu$ via the well-known relationships
$$
\lambda = \frac{E\nu}{(1+\nu)(1-2\nu)},\quad \mu = \frac{E}{2(1+\nu)}.
$$ 
We fix the
  Young's modulus $E = 10^{5}$ and vary the Poisson's ratio $\nu = 0.5
  - k$, $k = 1e-l$, for $l =\{1,2,5,9\}$. Locking would result in the
  displacement solution going to zero as $\nu$ approaches $0.5$. In
  Figure~\ref{fig:7} (left) we see that such behavior is not present,
  confirming the robustness of the method for almost incompressible
materials. In addition, a plot
  of the displacement magnitude along the top side of the square
  ($y=1$) for various choices of $k$ is shown in the
  Figure~\ref{fig:7} (right).  One can see that there is little 
  change in the displacement solution when $\nu \to 0.5$.}

\begin{figure}[ht!]
	\centering
	\begin{subfigure}[b]{0.65\textwidth}	
    		\includegraphics[width=0.25\textwidth]{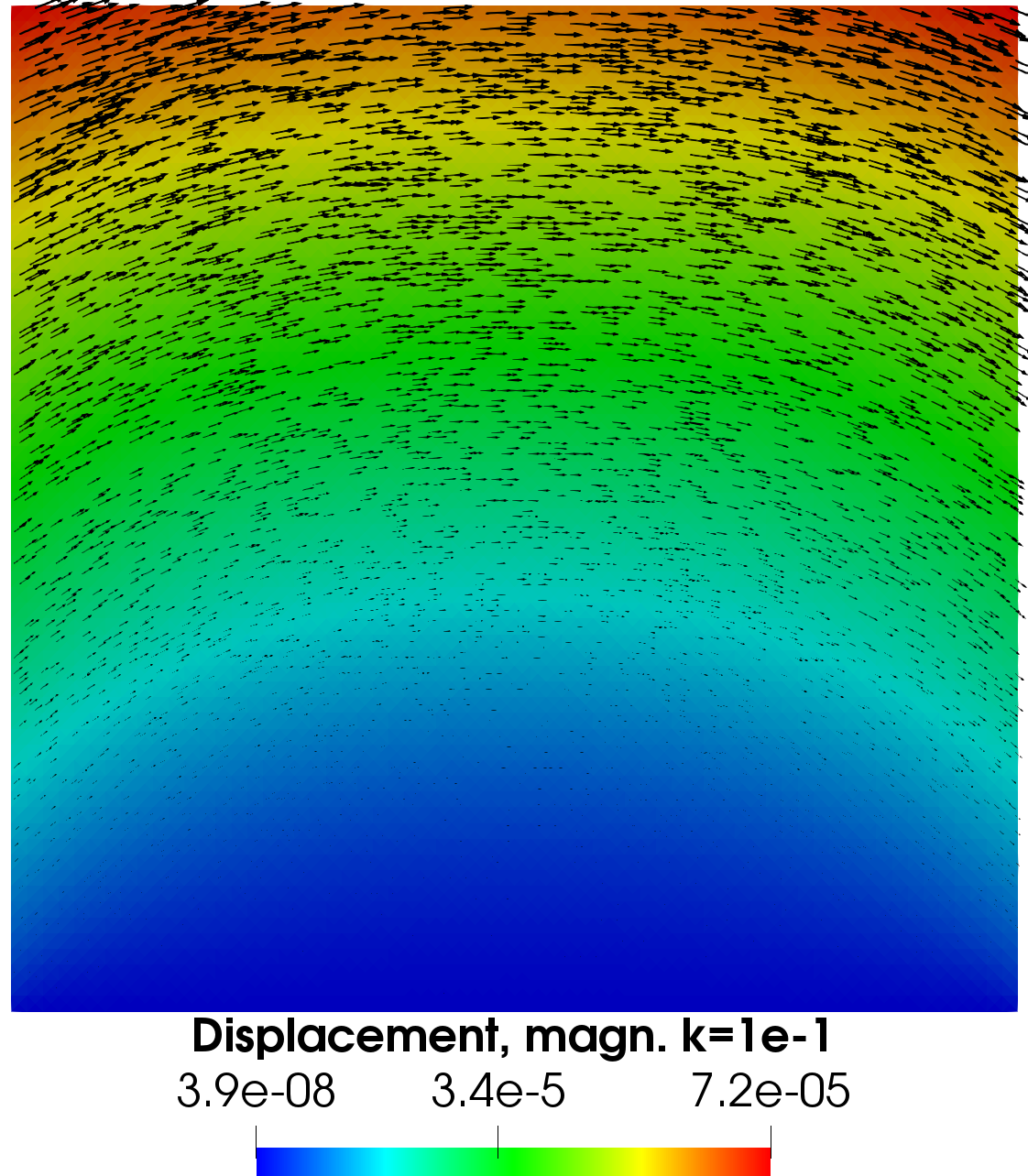}
    		\includegraphics[width=0.25\textwidth]{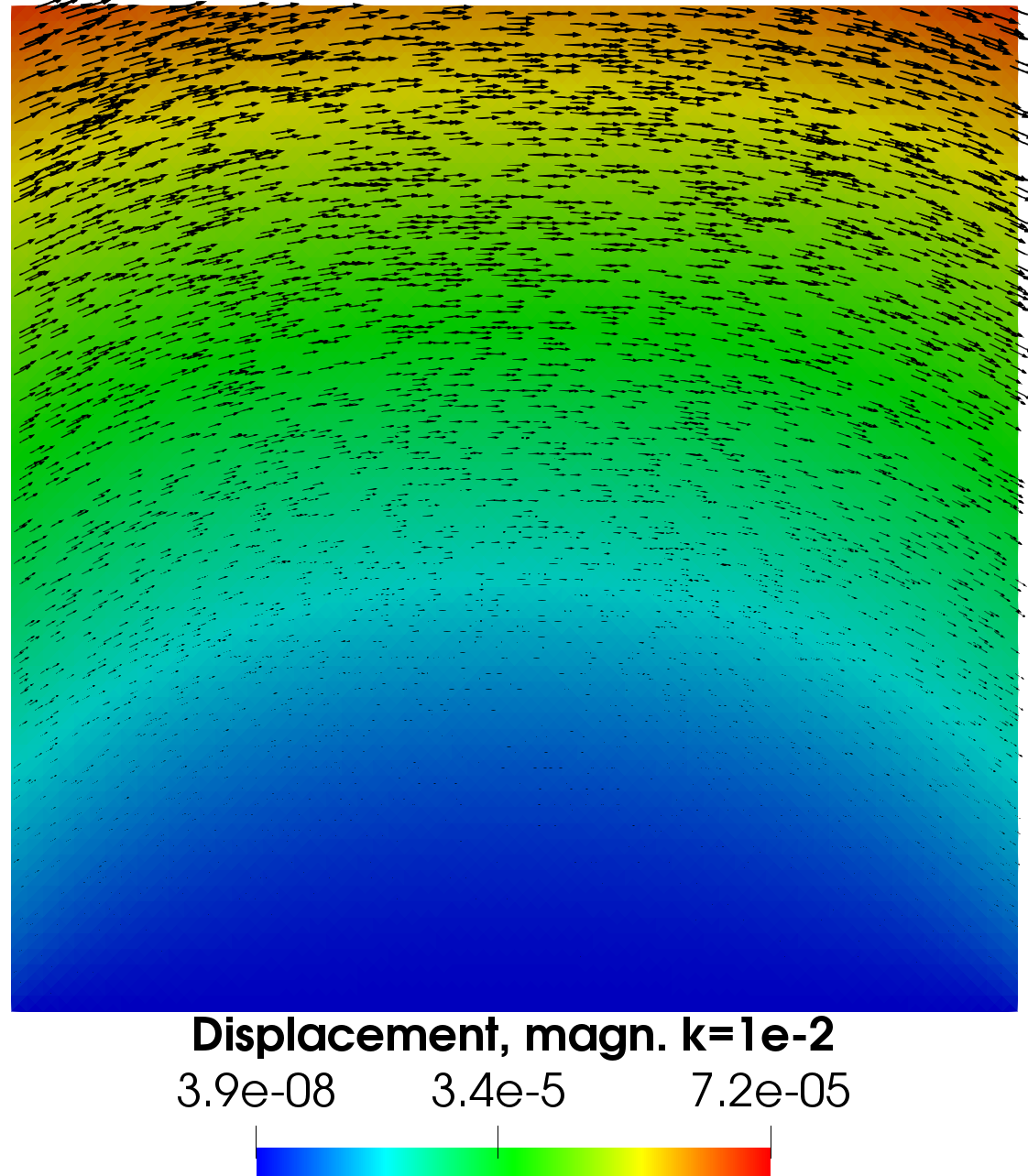}\\
		\includegraphics[width=0.25\textwidth]{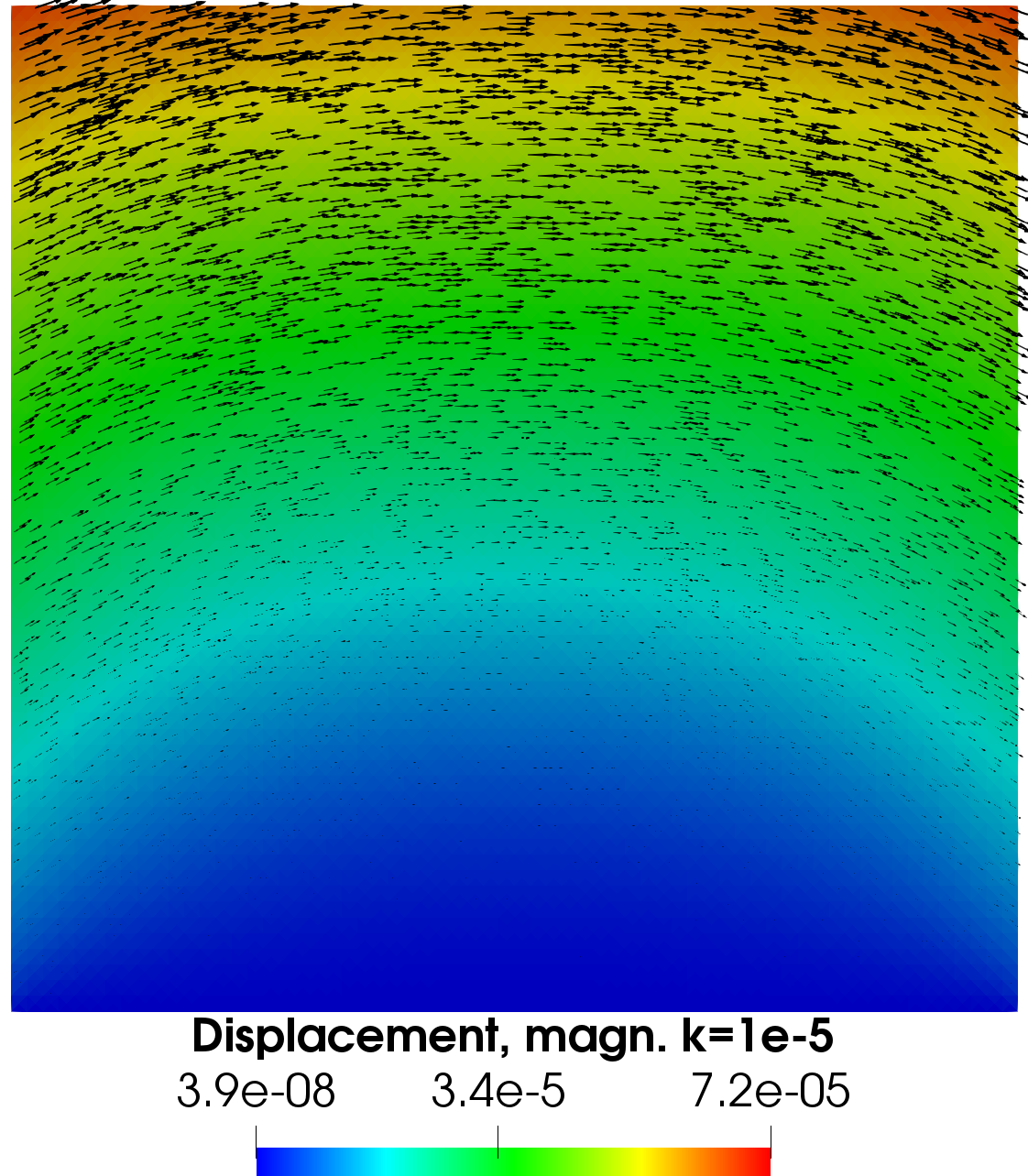}
    		\includegraphics[width=0.25\textwidth]{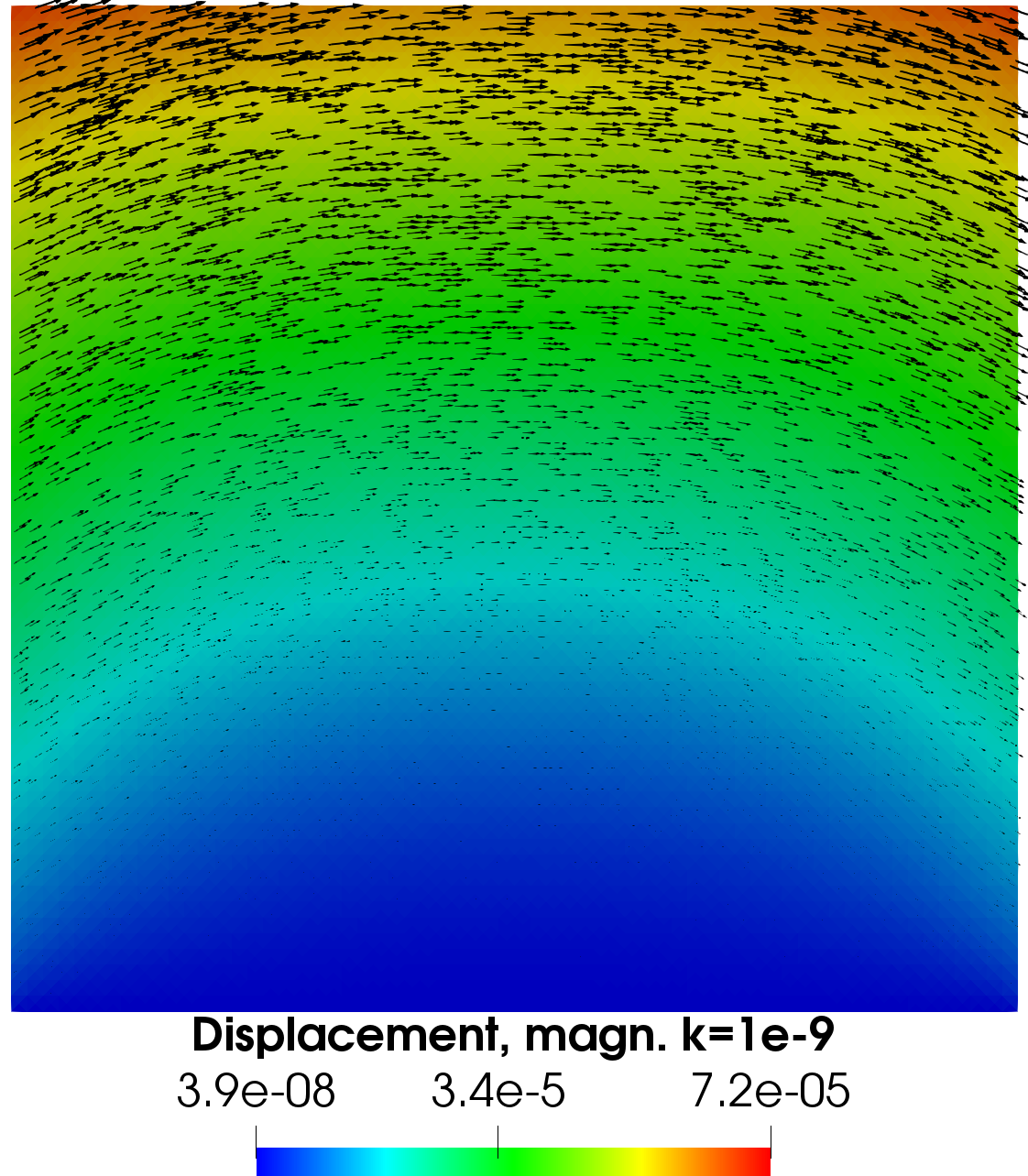}
    		\label{fig:7_1}
	\end{subfigure}		
	\hspace{-10em}
	\begin{subfigure}[b]{0.34\textwidth}
		\includegraphics[width=\textwidth]{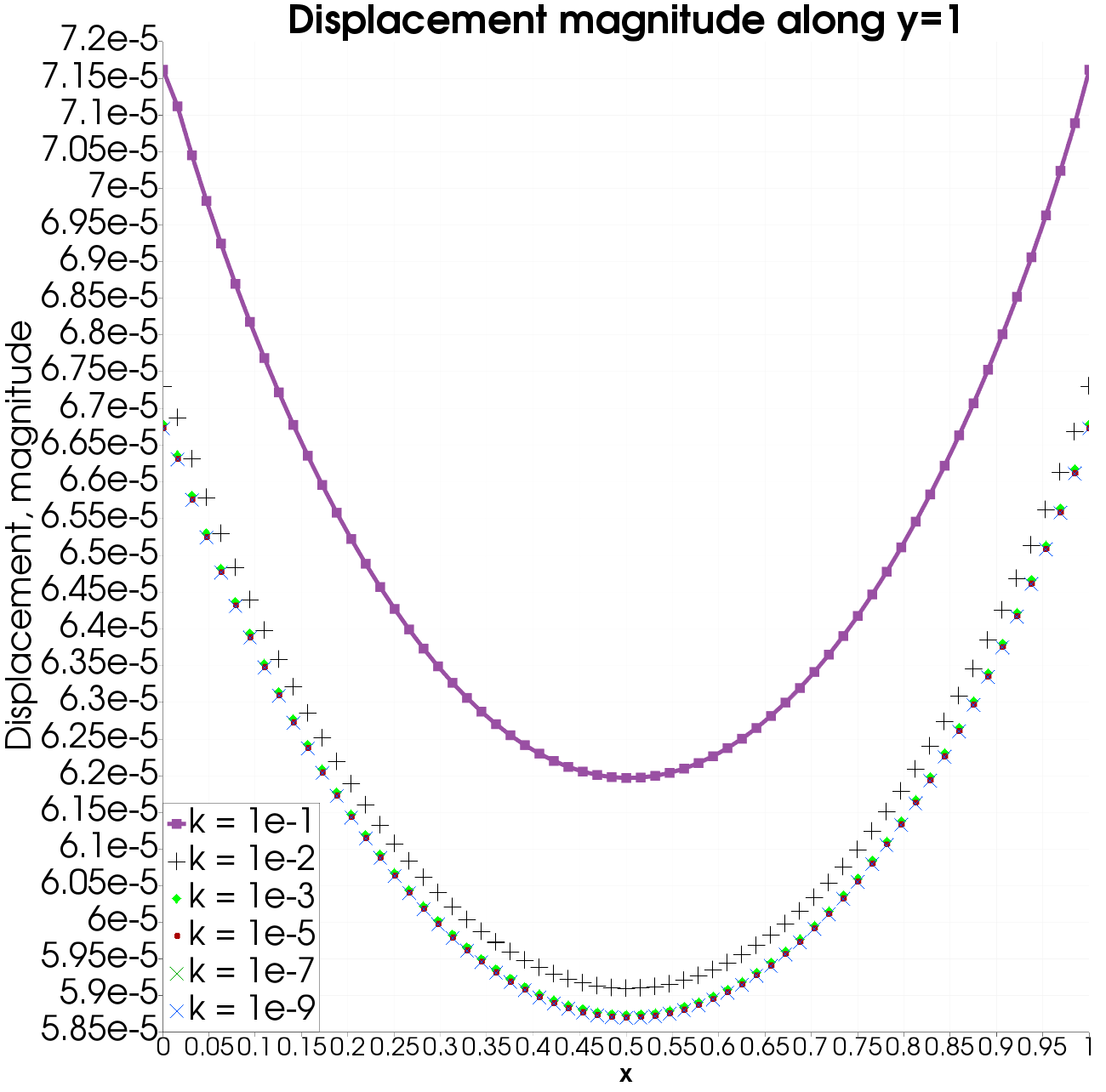}
		\label{fig:8_1}
	\end{subfigure}
	\caption{Computed displacement solutions for Example 4, $h=1/32$.}\label{fig:7}
\end{figure}


\section{Conclusion}
We presented two $\BDM_1$-based MFE methods with quadrature for
elasticity with weak stress symmetry on simplicial grids. The MSMFE-0
method reduces to a cell-centered scheme for displacements and
rotations, while the MSMFE-1 method reduces to a cell-centered scheme
for displacements only. To prove stability of the MSMFE-1 method, we
established a discrete inf-sup condition with quadrature for the
Stokes problem.  We showed that the resulting algebraic system for
each of the methods is symmetric and positive definite. We proved
first order convergence for all variables in their natural norms, as
well as second order convergence for the displacements at the cell
centers. The methods can also be developed on quadrilateral grids, which
is the subject of a forthcoming paper.

\bibliographystyle{abbrv}
\bibliography{msmfe-simpl}

\begin{thebibliography}{10}

\bibitem{aavat2002introduction}
I.~Aavatsmark.
\newblock An introduction to multipoint flux approximations for quadrilateral
  grids.
\newblock {\em Comput. Geosci.}, 6(3-4):405--432, 2002.
\newblock Locally conservative numerical methods for flow in porous media.

\bibitem{aavat1998discretization}
I.~Aavatsmark, T.~Barkve, O.~B{\o}e, and T.~Mannseth.
\newblock Discretization on unstructured grids for inhomogeneous, anisotropic
  media. {II}. {D}iscussion and numerical results.
\newblock {\em SIAM J. Sci. Comput.}, 19(5):1717--1736, 1998.

\bibitem{agelas2010convergence}
L.~Ag{\'e}las, C.~Guichard, and R.~Masson.
\newblock Convergence of finite volume {MPFA O} type schemes for heterogeneous
  anisotropic diffusion problems on general meshes.
\newblock {\em International Journal on Finite Volumes}, pages volume--7, 2010.

\bibitem{ADKWY}
T.~Arbogast, C.~N. Dawson, P.~T. Keenan, M.~F. Wheeler, and I.~Yotov.
\newblock Enhanced cell-centered finite differences for elliptic equations on
  general geometry.
\newblock {\em SIAM J.\ Sci.\ Comp.}, 19(2):404--425, 1998.

\bibitem{arbogast1997mixed}
T.~Arbogast, M.~F. Wheeler, and I.~Yotov.
\newblock Mixed finite elements for elliptic problems with tensor coefficients
  as cell-centered finite differences.
\newblock {\em SIAM J. Numer. Anal.}, 34(2):828--852, 1997.

\bibitem{arnold2005rectangular}
D.~N. Arnold and G.~Awanou.
\newblock Rectangular mixed finite elements for elasticity.
\newblock {\em Math. Models Methods Appl. Sci.}, 15(9):1417--1429, 2005.

\bibitem{arnold2015mixed}
D.~N. Arnold, G.~Awanou, and W.~Qiu.
\newblock Mixed finite elements for elasticity on quadrilateral meshes.
\newblock {\em Adv. Comput. Math.}, 41(3):553--572, 2015.

\bibitem{Arnold-etal-elast-nonconf-3d}
D.~N. Arnold, G.~Awanou, and R.~Winther.
\newblock Nonconforming tetrahedral mixed finite elements for elasticity.
\newblock {\em Math. Models Methods Appl. Sci.}, 24(4):783--796, 2014.

\bibitem{arnold1984peers}
D.~N. Arnold, F.~Brezzi, and J.~Douglas, Jr.
\newblock P{EERS}: a new mixed finite element for plane elasticity.
\newblock {\em Japan J. Appl. Math.}, 1(2):347--367, 1984.

\bibitem{arnold2007mixed}
D.~N. Arnold, R.~S. Falk, and R.~Winther.
\newblock Mixed finite element methods for linear elasticity with weakly
  imposed symmetry.
\newblock {\em Math. Comp.}, 76(260):1699--1723, 2007.

\bibitem{ArnWin-elast}
D.~N. Arnold and R.~Winther.
\newblock Mixed finite elements for elasticity.
\newblock {\em Numer. Math.}, 92(3):401--419, 2002.

\bibitem{Arnold-Winther-elast-nonconf-2d}
D.~N. Arnold and R.~Winther.
\newblock Nonconforming mixed elements for elasticity.
\newblock {\em Math. Models Methods Appl. Sci.}, 13(3):295--307, 2003.
\newblock Dedicated to Jim Douglas, Jr. on the occasion of his 75th birthday.

\bibitem{Awanou-rect-weak}
G.~Awanou.
\newblock Rectangular mixed elements for elasticity with weakly imposed
  symmetry condition.
\newblock {\em Adv. Comput. Math.}, 38(2):351--367, 2013.

\bibitem{VEM-elast}
L.~Beir\~{a}o~da Veiga, F.~Brezzi, and L.~D. Marini.
\newblock Virtual elements for linear elasticity problems.
\newblock {\em SIAM J. Numer. Anal.}, 51(2):794--812, 2013.

\bibitem{brezzi2008mixed}
D.~Boffi, F.~Brezzi, L.~F. Demkowicz, R.~G. Dur\'an, R.~S. Falk, and M.~Fortin.
\newblock {\em Mixed finite elements, compatibility conditions, and
  applications}, volume 1939 of {\em Lecture Notes in Mathematics}.
\newblock Springer-Verlag, Berlin; Fondazione C.I.M.E., Florence, 2008.
\newblock Lectures given at the C.I.M.E. Summer School held in Cetraro, June
  26--July 1, 2006, Edited by Boffi and Lucia Gastaldi.

\bibitem{boffi2009reduced}
D.~Boffi, F.~Brezzi, and M.~Fortin.
\newblock Reduced symmetry elements in linear elasticity.
\newblock {\em Commun. Pure Appl. Anal.}, 8(1):95--121, 2009.

\bibitem{Brenner-Sung}
S.~C. Brenner and L.-Y. Sung.
\newblock Linear finite element methods for planar linear elasticity.
\newblock {\em Math. Comp.}, 59(200):321--338, 1992.

\bibitem{brezzi1985two}
F.~Brezzi, J.~Douglas, Jr., and L.~D. Marini.
\newblock Two families of mixed finite elements for second order elliptic
  problems.
\newblock {\em Numer. Math.}, 47(2):217--235, 1985.

\bibitem{brezzi1991mixed}
F.~Brezzi and M.~Fortin.
\newblock {\em Mixed and hybrid finite element methods}, volume~15 of {\em
  Springer Series in Computational Mathematics}.
\newblock Springer-Verlag, New York, 1991.

\bibitem{Brezzi.F;Fortin.M;Marini.L2006}
F.~Brezzi, M.~Fortin, and L.~D. Marini.
\newblock Error analysis of piecewise constant pressure approximations of
  {D}arcy's law.
\newblock {\em Comput. Methods Appl. Mech. Eng.}, 195:1547--1559, 2006.

\bibitem{ciarlet2002finite}
P.~G. Ciarlet.
\newblock {\em The finite element method for elliptic problems}.
\newblock SIAM, 2002.

\bibitem{cockburn2010new}
B.~Cockburn, J.~Gopalakrishnan, and J.~Guzm\'an.
\newblock A new elasticity element made for enforcing weak stress symmetry.
\newblock {\em Math. Comp.}, 79(271):1331--1349, 2010.

\bibitem{Cockburn-Shi-HDG-elast}
B.~Cockburn and K.~Shi.
\newblock Superconvergent {HDG} methods for linear elasticity with weakly
  symmetric stresses.
\newblock {\em IMA J. Numer. Anal.}, 33(3):747--770, 2013.

\bibitem{dipietro2015hybrid}
D.~A. Di~Pietro and A.~Ern.
\newblock A hybrid high-order locking-free method for linear elasticity on
  general meshes.
\newblock {\em Comput. Methods Appl. Mech. Engrg.}, 283:1--21, 2015.

\bibitem{DiPietro-FV-elasticity}
D.~A. Di~Pietro, R.~Eymard, S.~Lemaire, and R.~Masson.
\newblock Hybrid finite volume discretization of linear elasticity models on
  general meshes.
\newblock In {\em Finite volumes for complex applications. {VI}. {P}roblems \&
  perspectives. {V}olume 1, 2}, volume~4 of {\em Springer Proc. Math.}, pages
  331--339. Springer, Heidelberg, 2011.

\bibitem{dipietro2015extension}
D.~A. Di~Pietro and S.~Lemaire.
\newblock An extension of the {C}rouzeix-{R}aviart space to general meshes with
  application to quasi-incompressible linear elasticity and {S}tokes flow.
\newblock {\em Math. Comp.}, 84(291):1--31, 2015.

\bibitem{edwards2002unstructured}
M.~G. Edwards.
\newblock Unstructured, control-volume distributed, full-tensor finite-volume
  schemes with flow based grids.
\newblock {\em Comput. Geosci.}, 6(3-4):433--452, 2002.
\newblock Locally conservative numerical methods for flow in porous media.

\bibitem{edwards1998finite}
M.~G. Edwards and C.~F. Rogers.
\newblock Finite volume discretization with imposed flux continuity for the
  general tensor pressure equation.
\newblock {\em Comput. Geosci.}, 2(4):259--290 (1999), 1998.

\bibitem{Galdi}
G.~P. Galdi.
\newblock {\em An introduction to the mathematical theory of the
  {N}avier-{S}tokes equations. {V}ol. {I}}.
\newblock Springer-Verlag, New York, 1994.
\newblock Linearized steady problems.

\bibitem{GG-elast-nonconf}
J.~Gopalakrishnan and J.~Guzm\'{a}n.
\newblock Symmetric nonconforming mixed finite elements for linear elasticity.
\newblock {\em SIAM J. Numer. Anal.}, 49(4):1504--1520, 2011.

\bibitem{gopalakrishnan2012second}
J.~Gopalakrishnan and J.~Guzm\'an.
\newblock A second elasticity element using the matrix bubble.
\newblock {\em IMA J. Numer. Anal.}, 32(1):352--372, 2012.

\bibitem{grisvard1985elliptic}
P.~Grisvard.
\newblock {\em Elliptic problems in nonsmooth domains}, volume~24 of {\em
  Monographs and Studies in Mathematics}.
\newblock Pitman (Advanced Publishing Program), Boston, MA, 1985.

\bibitem{Hansbo-Larson}
P.~Hansbo and M.~G. Larson.
\newblock Discontinuous {G}alerkin and the {C}rouzeix-{R}aviart element:
  application to elasticity.
\newblock {\em M2AN Math. Model. Numer. Anal.}, 37(1):63--72, 2003.

\bibitem{Horn-Johnson}
R.~A. Horn and C.~R. Johnson.
\newblock {\em Matrix analysis}.
\newblock Cambridge University Press, Cambridge, second edition, 2013.

\bibitem{Ing-Whe-Yot}
R.~Ingram, M.~F. Wheeler, and I.~Yotov.
\newblock A multipoint flux mixed finite element method on hexahedra.
\newblock {\em SIAM J. Numer. Anal.}, 48(4):1281--1312, 2010.

\bibitem{keilegavlen2017finite}
E.~Keilegavlen and J.~M. Nordbotten.
\newblock Finite volume methods for elasticity with weak symmetry.
\newblock {\em Int. J. Numer. Meth. Engng.}, 112(8):939--962, 2017.

\bibitem{klausen2006robust}
R.~A. Klausen and R.~Winther.
\newblock Robust convergence of multi point flux approximation on rough grids.
\newblock {\em Numer. Math.}, 104(3):317--337, 2006.

\bibitem{lee2016towards}
J.~J. Lee.
\newblock Towards a unified analysis of mixed methods for elasticity with
  weakly symmetric stress.
\newblock {\em Adv. Comput. Math.}, 42(2):361--376, 2016.

\bibitem{lions1972non}
J.-L. Lions and E.~Magenes.
\newblock {\em Non-homogeneous boundary value problems and applications. {V}ol.
  {II}}.
\newblock Springer-Verlag, New York-Heidelberg, 1972.
\newblock Translated from the French by P. Kenneth, Die Grundlehren der
  mathematischen Wissenschaften, Band 182.

\bibitem{LoggMardalEtAl2012a}
A.~Logg, K.-A. Mardal, G.~N. Wells, et~al.
\newblock {\em Automated Solution of Differential Equations by the Finite
  Element Method}.
\newblock Springer, 2012.

\bibitem{Jan-IJNME}
J.~M. Nordbotten.
\newblock Cell-centered finite volume discretizations for deformable porous
  media.
\newblock {\em Internat. J. Numer. Methods Engrg.}, 100(6):399--418, 2014.

\bibitem{nordbotten2015convergence}
J.~M. Nordbotten.
\newblock Convergence of a cell-centered finite volume discretization for
  linear elasticity.
\newblock {\em SIAM J. Numer. Anal.}, 53(6):2605--2625, 2015.

\bibitem{Qui-etal-HDG-elast}
W.~Qiu, J.~Shen, and K.~Shi.
\newblock An {HDG} method for linear elasticity with strong symmetric stresses.
\newblock {\em Math. Comp.}, 87(309):69--93, 2018.

\bibitem{riviere2000optimal}
B.~Rivi{\`e}re and M.~F. Wheeler.
\newblock Optimal error estimates for discontinuous {G}alerkin methods applied
  to linear elasticity problems.
\newblock {\em Comput. Math. Appl}, 46:141--163, 2000.

\bibitem{russell1983finite}
T.~F. Russell and M.~F. Wheeler.
\newblock Finite element and finite difference methods for continuous flows in
  porous media.
\newblock {\em The mathematics of reservoir simulation}, 1:35--106, 1983.

\bibitem{stenberg1984analysis}
R.~Stenberg.
\newblock Analysis of mixed finite elements methods for the {S}tokes problem: a
  unified approach.
\newblock {\em Math. Comp.}, 42(165):9--23, 1984.

\bibitem{stenberg1988family}
R.~Stenberg.
\newblock A family of mixed finite elements for the elasticity problem.
\newblock {\em Numer. Math.}, 53(5):513--538, 1988.

\bibitem{WG-elast}
C.~Wang, J.~Wang, R.~Wang, and R.~Zhang.
\newblock A locking-free weak {G}alerkin finite element method for elasticity
  problems in the primal formulation.
\newblock {\em J. Comput. Appl. Math.}, 307:346--366, 2016.

\bibitem{WheXueYot-nonsym}
M.~F. Wheeler, G.~Xue, and I.~Yotov.
\newblock A multipoint flux mixed finite element method on distorted
  quadrilaterals and hexahedra.
\newblock {\em Numer. Math.}, 121(1):165--204, 2012.

\bibitem{wheeler2006multipoint}
M.~F. Wheeler and I.~Yotov.
\newblock A multipoint flux mixed finite element method.
\newblock {\em SIAM J. Numer. Anal.}, 44(5):2082--2106, 2006.

\end{thebibliography}
\end{document}